\documentclass[reqno,11pt]{article}
\usepackage{amsthm}
\usepackage{a4wide}
\usepackage{color,eucal,enumerate,mathrsfs}
\usepackage[normalem]{ulem}
\usepackage{amsmath,amssymb,epsfig,bbm}
\numberwithin{equation}{section}

\usepackage{hyperref}

\usepackage[latin1]{inputenc}

\newcommand{\N}{\mathbb{N}}

\newcommand{\R}{\mathbb{R}}

\newcommand{\mm}{{\mbox{\boldmath$m$}}}

\newcommand{\ppi}{{\mbox{\boldmath$\pi$}}}

\newcommand{\sfd}{{\sf d}}

\newcommand{\Kliminf}{K\kern-3pt-\kern-2pt\mathop{\rm lim\,inf}\limits}  
\newcommand{\supp}{\mathop{\rm supp}\nolimits}   
\newcommand{\Lip}{\mathop{\rm Lip}\nolimits}          
\renewcommand{\d}{{\mathrm d}}

\newcommand{\restr}[1]{\lower3pt\hbox{$|_{#1}$}}

\newcommand{\eps}{\varepsilon}  
\newcommand{\nchi}{{\raise.3ex\hbox{$\chi$}}}

\newcommand{\lims}{\varlimsup}

\newcommand{\fr}{\penalty-20\null\hfill$\blacksquare$}                      

\newcommand{\prob}[1]{\mathscr P(#1)}                   
\newcommand{\e}{{\rm{e}}}                           

\renewcommand{\mm}{\mathfrak m}                                

\newcommand{\weakgrad}[1]{|\nabla #1|_w} 

\newcommand{\bd}{{\mathbf\Delta}}

\newcommand{\s}{{\rm S}}

\newcommand{\test}[1]{{\rm Test}(#1)}

\renewcommand{\div}{\mathbf{div}}

\newtheorem{theorem}{Theorem}[section]

\newtheorem{corollary}[theorem]{Corollary}
\newtheorem{lemma}[theorem]{Lemma}
\newtheorem{proposition}[theorem]{Proposition}

\newtheorem{definition}[theorem]{Definition}

\newtheorem{remark}[theorem]{Remark}

\newcommand{\sign}{{\rm sign}}
\renewcommand{\supp}{{\rm supp}}

\newcommand{\Test}{{\rm Test}}
\renewcommand{\div}{{\rm div}}
\newcommand{\bdiv}{{\rm{ \mathbf {div}}}}
\newcommand{\bDelta}{{ \mathbf {\Delta}}}

\newcommand{\Prob}{\mathscr P}
\renewcommand{\b}{{\rm b}}
\renewcommand{\weakgrad}[1]{|D#1|_w}

\title{A PDE approach to nonlinear potential theory in metric measure spaces}
\author{Nicola Gigli, Andrea Mondino}
\begin{document}
\maketitle

\begin{abstract}
We show that the tools recently introduced by the first author in \cite{Gigli12} allow to give a PDE description of $p$-harmonic functions in metric measure setting. Three applications are given: the first is about  new results on the sheaf property of harmonic functions, the second is a PDE proof of the fact that the composition of a subminimizer with a convex and non-decreasing function is again a subminimizer, and  the third is the fact that the Busemann function associated to a line is harmonic on infinitesimally Hilbertian $CD(0,N)$ spaces.
\end{abstract}

\tableofcontents

\section{Introduction}

The terminology `nonlinear potential theory in metric measure spaces' refers to the study  of real valued $p$-harmonic functions defined in the abstract setting of metric measure spaces and related topics. We refer to \cite{BjornBjorn11} for an overview of the subject and detailed bibliography. 

In the classical Euclidean setting there are two equivalent ways to formulate the statement `the function $g\in W^{1,p}(\Omega)$ is $p$-harmonic', being $\Omega\subset\R^d$ open and $p\in(1,\infty)$. One consists in requiring that $\nabla\cdot(|\nabla g|^{p-2}\nabla g)=0$ in the sense of distributions, the other in requiring that $g$ minimizes the $p$-energy, i.e.
\begin{equation}
\label{eq:class}
\int_\Omega|\nabla g|^p\,\d\mathcal L^d\leq \int_\Omega|\nabla(g+f)|^p\,\d\mathcal L^d,\qquad\forall f\in W^{1,p}_0(\Omega).
\end{equation}
The vector space structure of $\R^d$ plays no role for what concerns such equivalence: the same hold if $\Omega$ is an open subset of a smooth Riemannian manifold, because the only thing needed to pass from a formulation to the other  is a smooth structure and integration by parts.

\bigskip

To approach a definition of $p$-harmonic function on  a non-smooth structure requires some work, in particular the first thing to do is to give the definition of Sobolev space of real valued functions defined on a metric measure space $(X,\sfd,\mm)$. Several equivalent definitions have been proposed (by Cheeger \cite{Cheeger00}, Shanmugalingam \cite{Shanmugalingam00} and the first author together with Ambrosio and Savar\'e \cite{Ambrosio-Gigli-Savare11}, \cite{Ambrosio-Gigli-Savare-pq}, the latter recalled in Section \ref{se:sobclass}), all of them having in common the fact that for a function $g\in W^{1,p}(X,\sfd,\mm)$ it is not defined its distributional gradient, but only its modulus $\weakgrad g$ typically called minimal generalized upper gradient or minimal weak upper gradient (although being this object defined in duality with the distance, it is naturally the norm of a cotangent vector rather than of a tangent one, thus it would be more proper to call it minimal generalized/weak upper `differential' whence the notation with `$D$' in place of `$\nabla$'). The definition of  $W^{1,p}(X,\sfd,\mm)$ can be naturally localized to obtain the space of Sobolev functions $W^{1,p}(\Omega)$ and $W^{1,p}_0(\Omega)$ defined on an open set $\Omega\subset X$.

With such space at disposal the variational formulation of $p$-harmonic functions, in this setting called $p$-minimizers, can be given: one says that $g\in W^{1,p}(\Omega)$ is a $p$-minimizer provided
\begin{equation}
\label{eq:intro}
\int_\Omega\weakgrad g^{p}\,\d\mm\leq \int_\Omega\weakgrad{(g+f)}^p\,\d\mm,\qquad\forall f\in W^{1,p}_0(\Omega),
\end{equation}
and nonlinear potential theory on metric measure spaces has been built on these ground. Quite surprisingly, assuming only completeness, a doubling condition on the measure and the validity of a local weak Poincar\'e inequality and despite the lack of a PDE characterization of $p$-minimizers, the theory has been pushed quite far. For instance the Harnack inequality, the strong maximum principle and several regularity results have been obtained, see \cite{BjornBjorn11} for a detailed overview of the subject.

\bigskip

Aim of this paper is to show that also a genuine PDE characterization of $p$-minimizers  can be given in this abstract framework.  The approach  we propose is independent from analysis in charts. In particular, the sort of PDE that we are going to define is not linked to the differential structure of metric measure spaces built by Cheeger in \cite{Cheeger00}. This structure can be used to define a (chart-dependent) notion of differential of a Sobolev function as an a.e. well defined vector in some $\R^N$. Then  one can define the scalar product $Df\cdot Dg$ of two Sobolev functions as the scalar product of these vectors as elements of $\R^N$. Studies about the regularity of the resulting notion of harmonic maps  have been done in \cite{KRS03} (see also the more recent paper \cite{J11} about solutions of the corresponding Poisson equation). One clear advantage of this approach is that one always obtains a bilinear map $(f,g)\mapsto Df\cdot Dg$ which yields - by integration by parts - a linear Laplacian. The drawback is that, being chart-dependent, $Df\cdot Dg$ is not intrinsically defined, and thus its link with the geometry of the space is not so evident. In particular, in general it is not true that $Df\cdot Df=\weakgrad f^2$ (only a two sided bound holds) and thus minimizers of $\int_\Omega\weakgrad f^2\,\d\mm$ are in general not the same as   the minimizers of $\int_\Omega Df\cdot Df \,\d\mm$.

\bigskip

Aiming for a  PDE description of minimizers in \eqref{eq:intro}, we will then   proceed differently. The key tool that we will use  is the definition of `differential of a function $f$ applied to the gradient of a function $g$' for Sobolev $f,g$ which has been proposed by the first author in \cite{Gigli12} and that we now describe. To fix the ideas, let us work for the moment on the space $(\R^d,\|\cdot-\cdot\|,\mathcal L^d)$, where $\|\cdot\|$ is a strictly convex norm. The differential $Df(x)$ of the smooth function $f:\R^d\to\R$ at a point $x$ is the cotangent vector defined by 
\[
Df(x)(v):=\lim_{t\to 0}\frac{f(x+tv)-f(x)}{t},
\]
for all tangent vectors $v\in T_x\R^d\sim\R^d$. By definition, the differential linearly depends on the function and its norm is computed w.r.t. the dual norm $\|\cdot\|_*$ of $\|\cdot\|$. The gradient $\nabla f(x)$ of $f$ as $x$ is the tangent vector $w$ defined by $\|Df(x)\|_*=\|w\|$ and $Df(x)(w)=\|Df(x)\|_*^2$. Given that for any $v\in T_x\R^d\sim\R^d$ it holds 
\begin{equation}
\label{eq:grad0}
Df(x)(v)\leq \|Df(x)\|_*\|v\|\leq\frac12\|Df(x)\|_*^2+\frac12\|v\|^2,
\end{equation}
it is easy to see that $w$ is the gradient $\nabla f(x)$ of $f$ at $x$ if and only if
\begin{equation}
\label{eq:grad}
Df(x)(w)\geq \frac12\|Df(x)\|_*^2+\frac12\|w\|^2.
\end{equation}
A simple compactness argument shows that $\nabla f(x)$ exists and the strict convexity of $\|\cdot\|$ ensures uniqueness. Observe however that in general the gradient does not depend linearly from the function (this is the case if and only if the norm comes from a scalar product). Being the gradient a tangent vector, its norm is computed w.r.t. $\|\cdot\|$. We remark that it holds $\|Df(x)\|_*=\|\nabla f(x)\|$ so that we can't really distinguish differentials and   gradients by only looking at their norms: the crucial algebraic difference is instead the fact that the former is linear on the function, while the latter is not. 

Now we claim that
\begin{equation}
\label{eq:scambio}
Df(\nabla g)(x)=\lim_{\eps\to 0}\frac{\|D(g+\eps f)(x)\|_*^2-\|Dg(x)\|_*^2}{2\eps}.
\end{equation}
Indeed, by the very definitions we have
\[
D(g+\eps f)(\nabla g)(x)\leq \frac12\|D(g+\eps f)(x)\|_*^2+\frac12\|\nabla g(x)\|^2,\qquad\forall \eps\in\R,
\]
and
\[
Dg(\nabla g)(x)\geq \frac12\|Dg(x)\|_*^2+\frac12\|\nabla g(x)\|^2.
\]
Subtract the second inequality from the first, divide by $\eps>0$ (resp. $\eps<0$) and let $\eps\downarrow 0$ (resp. $\eps\uparrow 0$) to obtain
\[
\begin{split}
Df(\nabla g)(x)&\leq \lim_{\eps\downarrow 0}\frac{\|D(g+\eps f)(x)\|_*^2-\|Dg(x)\|_*^2}{2\eps},\\
Df(\nabla g)(x)&\geq \lim_{\eps\uparrow 0}\frac{\|D(g+\eps f)(x)\|_*^2-\|Dg(x)\|_*^2}{2\eps}.
\end{split}
\]
Then notice that the strict convexity of $\|\cdot\|$ is equivalent to the differentiability of $\|\cdot\|_*$ to conclude. The interesting fact about the identity \eqref{eq:scambio} is that it defines the value of $Df(\nabla g)(x)$ starting only from the notion of norm of differential.

If the norm is not strictly convex the situation complicates a bit, because the gradient of a smooth function is not anymore uniquely defined. This is best understood with an example. Endow $\R^2$ with the $L^\infty$ norm  and consider the function $f:\R^2\to\R$ given by $f(x_1,x_2):=x_1$. In this case, all the vectors $v$ of the kind $v=(1,v_2)$ with $v_2\in[-1,1]$ can be called gradient of $f$ at, say, $(0,0)$. Indeed all of them have norm $1$ and the derivative of $f$ at $(0,0)$ along any of them is 1, 1 being also the (dual) norm of the differential of $f$.

Thus in general we must work with a multivalued gradient. The definition can be given as  before: $w\in \nabla f(x)$ provided inequality \eqref{eq:grad} holds (notice that inequality \eqref{eq:grad0} remains valid in this higher generality). Being the gradient multivalued, we can't hope anymore to define a uniquely valued map $Df(\nabla g):\R^d\to\R$ and the best we can do is to consider its maximal and minimal values
\[
\begin{split}
D^+f(\nabla g)(x)&:=\max_{w\in\nabla g(x)}Df(x)(w),\\
D^-f(\nabla g)(x)&:=\min_{w\in\nabla g(x)}Df(x)(w).
\end{split}
\]
With arguments similar to those used to prove \eqref{eq:scambio} in the case of strictly convex norms, one can see that it holds
\begin{equation}
\label{eq:scambio2}
\begin{split}
D^+f(\nabla g)(x)&=\lim_{\eps\downarrow0}\frac{\|D(g+\eps f)(x)\|_*^2-\|Dg(x)\|_*^2}{2\eps},\\
D^-f(\nabla g)(x)&=\lim_{\eps\uparrow0}\frac{\|D(g+\eps f)(x)\|_*^2-\|Dg(x)\|_*^2}{2\eps}.
\end{split}
\end{equation}

On metric measure spaces $(X,\sfd,\mm)$ we don't have an a priori notion of differential and gradient of Sobolev functions, but as said we have a notion of `modulus of the distributional differential'. Therefore we can use the right hand sides of \eqref{eq:scambio2} to define the $\mm$-a.e. value of $D^\pm f(\nabla g)$ for Sobolev $f,g$. The existence of the limits is ensured by the $\mm$-a.e. convexity of the map $\eps\mapsto\weakgrad{(g+\eps f)}^2$.

It turns out that for fixed $g$, the map $f\mapsto D^+f(\nabla g)$ (resp. $f\mapsto D^-f(\nabla g)$) is positively 1-homogeneous, $\mm$-a.e. convex (resp. $\mm$-a.e. concave) and 1-Lipschitz in the sense that $|D^\pm f(\nabla g)-D^\pm\tilde f(\nabla g)|\leq \weakgrad{(f-\tilde f)}\weakgrad g$ $\mm$-a.e.. Conversely, for fixed $f$ the maps $g\mapsto D^\pm f(\nabla g)$ are positively 1-homogeneous and have some general semicontinuity property (Proposition \ref{prop:ConvHomog}). 

It is important to underline that we are not defining, nor we will, what are the differential of $f$ and the gradient of $g$, but only what is the value of `the differential of $f$ applied to the gradient of $g$', which is all one needs to integrate by parts (for a proposal of what is the gradient of a function see Definition 3.7 in \cite{Gigli12}). The fact that this is a reasonable definition comes from the validity of the chain rules
\[
\begin{split}
D^\pm (\varphi\circ f)(\nabla g)&=\varphi'\circ f D^{\pm\sign(\varphi'\circ f)}f(\nabla g),\\
D^\pm f(\nabla (\psi\circ  g))&=\psi'\circ g D^{\pm\sign(\psi'\circ g)}f(\nabla g),
\end{split}
\]
where $\varphi,\psi:\R\to\R$ are Lipschitz and of the Leibniz rules for differentials
\[
\begin{split}
D^+(f_1f_2)(\nabla g)&\leq f_1 D^{s_1}f_2(\nabla g)+f_2D^{s_2}f_1(\nabla g),\\
D^-(f_1f_2)(\nabla g)&\geq f_1 D^{-s_1}f_2(\nabla g)+f_2D^{-s_2}f_1(\nabla g),
\end{split}
\]
where $s_i$ is the sign of $f_i$, $i=1,2$. See Subsection \ref{se:dpmfg} for the precise statements. In connection with this calculus, there are two interesting particular cases.
\begin{itemize}
\item \underline{Infinitesimally strictly convex spaces,} i.e. spaces which resemble $\R^d$ with a strictly convex norm. This can be read at the abstract level as those spaces for which $D^+f(\nabla g)=D^-f(\nabla g)$ for any Sobolev $f,g$. In this case their common value will be denoted by $Df(\nabla g)$ and the calculus rules simplify, as the map $f\mapsto Df(\nabla g)$ is linear and the Leibniz rule for differentials holds as equality.
\item \underline{Infinitesimally Hilbertian spaces,} i.e. spaces which resemble $\R^d$ with a norm coming from a scalar product. This can be read in abstract by requiring that the Sobolev space $W^{1,2}(X,\sfd,\mm)$ is an Hilbert space (in general it is only Banach). With some work it can be proved that these spaces are infinitesimally strictly convex and that the object $Df(\nabla g)$ is symmetric in $f,g$, thus in particular it is bilinear and the Leibniz rule holds also for the gradients.
\end{itemize}
In connection to this, it is worth to underline that there are two sources of nonlinearity in nonlinear potential theory in metric measure spaces: one is due to the exponent $p$ which, when different from 2, causes the $p$-Laplacian to be nonlinear even in the Euclidean setting, the other is due to the nonlinear dependence of gradients from functions, well known in a Finsler context, which in general let the Laplacian be a nonlinear operator  (see for instance \cite{Shen98}).

All this comes from \cite{Gigli12} (see also \cite{Ambrosio-Gigli-Savare11bis} for the original discussion on infinitesimal Hilbertianity, there mentioned as `spaces with quadratic Cheeger energies'), where these notions have been used to define the distributional Laplacian and prove Laplacian comparison estimates on spaces with Ricci curvature bounded from below.

\bigskip

As mentioned, in this paper we show that these calculus tools allow to give a PDE description of $p$-minimizers   on doubling spaces supporting a weak local Poincar\'e inequality. The crucial definition that we give is that of distributional divergence: in short, given a Sobolev function $g$ on an open set $\Omega\subset X$ with $\weakgrad g\in L^p(\Omega,\mm)$  and $h\in L^q(\Omega,\mm)$ with $p,q\in(1,\infty)$ conjugate exponents, we say that $h\nabla g$ is in the domain of the divergence, and write $h\nabla g\in D(\bdiv,\Omega)$, provided there exists a measure $\mu$ on $\Omega$ such that the inequalities
\begin{equation}
\label{eq:divdef}
-\int_\Omega hD^{{\rm sign}(h)}f(\nabla g)\,\d\mm\leq \int_\Omega f\,\d\mu\leq\int_\Omega hD^{-{\rm sign}(h)}f(\nabla g)\,\d\mm,
\end{equation}
hold for any Lipschitz and compactly supported function $f$ on $\Omega$. In this case we write $\mu\in\bdiv(h\nabla g)\restr\Omega$, the notation in bold standing to remember that the divergence so defined is a measure, potentially multivalued. It is certainly unusual to deal with a multivalued divergence, but the lack of a single valued notion of `differential applied to gradient' implies that we can't write the integration by parts formula as equality, so that the best thing we can do is to ask $\int_\Omega f\,\d\mu$ to be between the maximal and minimal corresponding integrated values of $h D^\pm f(\nabla g)$, as in \eqref{eq:divdef}.
 
As for the object $D^\pm f(\nabla g)$, it should be noted that the expression $\bdiv(h\nabla g)$ is purely formal in the sense that we don't have a definition for $h\nabla g$ nor for $\bdiv$. The justification of the notation comes from the fact that the expected calculus rules hold. For instance we prove the chain rule
\[
\bdiv(h\nabla (\varphi\circ g))\restr\Omega=\bdiv(h\varphi'\circ g\,\nabla g)\restr\Omega,
\]
the Leibniz rule on infinitesimally strictly convex spaces
\[
\bdiv(h_1h_2\nabla g)\restr\Omega=Dh_1(\nabla g)h_2\mm\restr\Omega+ h_1\bdiv(h_2\nabla g),
\]
and the Leibniz rule for gradients on infinitesimally Hilbertian spaces
\[
\bdiv(h\nabla(g_1g_2))=\bdiv(hg_1\nabla g_2)+\bdiv(hg_2\nabla g_1).
\]
All these formulas are proved under quite general and natural assumptions on the functions involved, see Subsection \ref{se:calcrule}.

Then in Section \ref{se:main} we prove the main result of the paper, Theorem \ref{thm:Solution}, namely that a function $g$ on $\Omega$ is a $p$-minimizer if and only if $\weakgrad g^{p-2}\nabla g\in D(\bdiv,\Omega)$ with $0\in \bdiv(\weakgrad g^{p-2}\nabla g)\restr\Omega$. We also study the problem for $p$-subminimizers (resp. superminimizers) i.e. functions $g$ such that 
\[
\int_\Omega\weakgrad g^p\,\d\mm\leq \int_\Omega\weakgrad{(g+f)}^p\,\d\mm,
\]
holds for all non-positive (rep. non-negative) $f\in W^{1,p}_0(\Omega)$. In the Euclidean setting, the corresponding PDE characterization is $\nabla\cdot(|\nabla g|^{p-2}\nabla g)\geq 0$ (resp. $\leq 0$). The same result holds in the abstract setting provided we assume infinitesimal strict convexity, see Corollary \ref{cor:SupSolSIC} for the result and the discussion before it for a comment about such assumption.

In Section \ref{se:appl} we give three applications of this results.
\begin{itemize}
\item[1)] We employ the local nature of PDE's to prove that on infinitesimally strictly convex spaces $g$ is a $p$-subminimizer on $\Omega=\cup_i\Omega_i$ if and only if it is a $p$-subminimizer in each of the $\Omega_i$'s, thus answering in this case to the Open Problems 9.22 and 9.23 in \cite{BjornBjorn11}, see Proposition \ref{prop:sheaf}. We weren't able to drop the assumption on infinitesimals strict convexity not even in the case of $p$-minimizers, see Remark \ref{re:mah} for comments in this direction.
\item[2)] We give a new proof, based on PDE techniques rather than on variational methods, of the fact that the composition of a subminimizer with a convex and non-decreasing function is again a subminimizer. See Proposition \ref{prop:compsub}.
\item[3)] We show that the Busemann function associated to a line on an infinitesimally Hilbertian $CD(0,N)$ space is harmonic, see Proposition \ref{prop:lapbus}. Here we somehow invert the point of view and rather than using the differential calculus developed to prove statements concerning nonlinear potential theory, we employ the maximum principle proved in this latter setting to deduce a new PDE result. Indeed, in \cite{Gigli12} it has been proved that the Busemann function $\b$ associated to an half-line on a $CD(0,N)$ space satisfies $\bdiv(\nabla \b)\leq 0$ (under some assumptions on the space which are fulfilled in the infinitesimal Hilbertian case). This means that for the two Busemann functions $\b^+,\b^-$ associated to a line we know that $\bdiv(\nabla \b^+),\bdiv(\nabla \b^-)\leq 0$ and that $\b^++\b^-$ has a global minimum (see Subsection \ref{se:bus}). According to what is known in the smooth case, we would like to deduce that $\b^++\b^-$ is constant, which is typically proved via the strong maximum principle. The very same thing can be proved in the non-smooth setting once we know the validity of the strong maximal principle, which we do according to Theorem 9.13 in \cite{BjornBjorn11}. 

About the inequality $\bdiv(\nabla \b)\leq 0$, it is worth to underline that it is proved via means that have nothing to do, in principle, with nonlinear potential theory. Indeed, the technique used to get it is related to an `horizontal' derivation (typical in the mass transport context) rather than to a `vertical' one (more common in Sobolev analysis), see Section 3.2. in \cite{Gigli12} and in particular Theorem 3.10 for a discussion on these topics. 
\end{itemize}

\bigskip

\noindent\emph{In preparing this paper the authors have been partially supported by ERC ADG GeMeThNES}

\section{Preliminaries}
\subsection{Metric measure spaces}
Throughout all the paper $(X,\sfd)$ will be a complete and separable metric space. We denote by $B_r(x)$ the open ball of center $x \in X$ and radius $r>0$. $C([0,1],X)$ is the complete and separable metric  space of continuous curves from $[0,1]$ with values in $X$ equipped with the $\sup$ norm. 

A curve $\gamma \in C([0,1], X)$ is said to be absolutely continuous if there exists a function $f\in L^1([0,1])$ such that
\begin{equation}\label{def:ACcurve}
\sfd(\gamma_t,\gamma_s)\leq \int_s^t f(r) \;\d r, \quad \forall t,s \in [0,1],\ \textrm{s.t. } t<s.
\end{equation} 
The set of absolutely continuous curves from $[0,1]$ to $X$ will be denoted by $AC([0,1],X)$. More generally if the function $f$ in \eqref{def:ACcurve} belongs to $L^q([0,1])$, $q \in [1,\infty]$, $\gamma$ is said $q$-absolutely continuous, and $AC^q([0,1],X)$ is the corresponding set of $q$-absolutely continuous curves. Recall (see for example Theorem 1.1.2 in \cite{Ambrosio-Gigli-Savare08}) that if $\gamma \in AC^q([0,1],X)$ then the limit
$$\lim_{h \to 0} \frac{\sfd(\gamma_{t+h},\gamma_t)}{|h|} $$
exists for a.e. $t \in [0,1]$. Such function is called \emph{metric speed} or \emph{metric derivative}, is denoted by $|\dot{\gamma}_t|$ and is the minimal (in the a.e. sense) $L^q$ function which can be chosen as $f$ in the right hand side of \eqref{def:ACcurve}. 

For every $t \in [0,1]$, we define the \emph{evaluation map} $\e_t:C([0,1],X)\to X$ as
$$\e_t(\gamma)=\gamma_t, \qquad\qquad \forall \gamma \in C([0,1],X). $$

For $f:X\to \R$ the \emph{local Lipschitz constant} $|D f|:X\to [0,\infty]$ is defined by 
\[
|Df|(x):=\limsup_{y\to x} \frac{|f(y)-f(x)|}{\sfd(x,y)}, \quad \text{ if $x$ is not isolated, $0$ otherwise. }
\]

Given a Borel measure $\sigma$ on $X$,  $\supp(\sigma)$ is the support of $\sigma$, i.e. the smallest closed set on which $\sigma$ is concentrated. We denote by $\prob X$ the set of Borel probability measures on $X$.
\\

Endowing the metric space $(X,\sfd)$ with a measure $\mm$,  we get a so called \emph{metric measure space} $(X,\sfd, \mm)$. Throughout all the paper we will assume that 
\begin{equation}
\label{eq:mms}
\begin{split}
(X,\sfd)& \text{ is  a complete separable metric space and} \\
\mm & \text{ is a Borel non negative and doubling  measure on } X,
\end{split}
\end{equation}
where  doubling means  that for some constant $C>0$ it holds
\[
\mm(B_{2r}(x))\leq C\mm(B_r(x)),\qquad\forall x\in X,\ r>0.
\]

Let $p_0\geq 1$. We say that $(X,\sfd,\mm)$ supports a weak local $(1,p_0)$-Poincar\'e inequality (or more briefly a $p_0$-Poincar\'e inequality) if there exists constants $C_{PI}$ and $\lambda\geq 1$ such that for all  $x\in X$, $r>0$ and Lipschitz functions $f:X\to\R$ it holds
\begin{equation}\label{LPI}
\frac1{\mm(B_r(x))}\int_{B_r(x)} |f-f_{B_r(x)}| \, \d \mm \leq C_{PI}\, 2r \left(\frac{1}{\mm(B_{\lambda r}(x))} \int_{B_{\lambda r}(x)} |D f|^{p_0} \, \d \mm \right)^{\frac{1}{p_0}},
\end{equation}
where $f_{B_r(x)}:=\frac1{\mm(B_r(x))}\int _{B_r(x)} f \, \d \mm$. We remark that typically the Poincar\'e inequality is required to hold for integrable functions $f$ and upper gradients $G$ (see for instance Definition 4.1 in \cite{BjornBjorn11}), rather than for Lipschitz functions and their local Lipschitz constant. It is obvious that the second formulation implies the one we gave, but also the converse implication holds, as a consequence of the density in energy of Lipschitz functions in the Sobolev spaces proved in \cite{Ambrosio-Gigli-Savare-pq} for the case $p_0>1$ and in \cite{Ambrosio-DiMarino} for $p_0=1$. Therefore the definition we chose is equivalent to the standard one.
\begin{remark}{\rm
In this paper we will mostly work on doubling spaces supporting some Poincar\'e inequality, but we remark that actually several results are true on general complete separable metric spaces endowed with a locally finite measure $\mm$. Indeed, the algebra behind our main results, which is the one recalled in Section \ref{se:dpmfg}, remains true in this higher generality, as showed in \cite{Gigli12}. The choice to work with doubling\&Poincar\'e is motivated by the following facts:
\begin{itemize}
\item the goal of this paper is to provide a link between the theory developed in \cite{Gigli12} and nonlinear potential theory, and the latter is typically developed in doubling\&Poincar\'e spaces,
\item these assumptions greatly simplify the exposition thanks to the strong density results of Lipschitz functions in Sobolev spaces (recalled in Theorem \ref{thm:approxLip}) and to the independence of the $p$-minimal weak upper gradient on $p$ (see Theorem \ref{thm:weakgradp}). This latter fact in particular will allow us to define the objects $D^\pm f(\nabla g)$ (which are the non-smooth analogous of `the differential of $f$ applied to the gradient of $g$') without referring to a particular Sobolev exponent $p$.
\end{itemize}
\vspace{-1cm}}\fr\end{remark}

\subsection{Sobolev classes}\label{se:sobclass}
In this subsection we recall the definition of Sobolev classes $\s^p(X,\sfd,\mm)$, which are the metric-measure analogous of the spaces of functions having distributional gradient in $L^p$ when the ambient space is the Euclidean one, regardless of any integrability assumption on the function themselves. In particular, the Sobolev space $W^{1,p}(X,\sfd,\mm)$ will be defined as $L^p(X,\mm)\cap\s^p(X,\sfd,\mm)$. Different approaches to these spaces have been proposed in the literature, most of them being equivalent (see \cite{Ambrosio-Gigli-Savare-pq} for a discussion about this topic) here we follow the approach introduced in \cite{Ambrosio-Gigli-Savare11} and \cite{Ambrosio-Gigli-Savare11bis} (see also  \cite{Gigli12} for a presentation closer to the one given here).

\begin{definition}[$q$-test plan]\label{def:testplan}
Let $(X,\sfd,\mm)$ be as in \eqref{eq:mms} and $\ppi \in \Prob(C[0,1],X)$. We say that $\ppi$ has bounded compression if there exists $C>0$ such that
$$(e_t)_{\sharp} \ppi \leq C \mm, \quad \forall t \in [0,1].$$
For $q \in (1, \infty)$ we say that $\ppi$ is a \emph{$q$-test plan} if it has bounded compression, is concentrated on $AC^q([0,1],X)$ and 
$$\iint_0^1 |\dot{\gamma}_t|^q \, \d t \; \ppi(\gamma)<\infty. $$ 
\end{definition}

\begin{definition}[Sobolev classes]\label{def:Sobolev}
Let $(X,\sfd,\mm)$ be as in \eqref{eq:mms}, $p\in (1,\infty)$ and $q$ the conjugate exponent. A Borel function $f:X\to \R$ belongs to the Sobolev class $\s^p(X,\sfd,\mm)$ (resp. $\s^p_{loc}(X,\sfd,\mm)$) if there exists a function $G\in L^p(X,\mm)$ (resp. in $L^p_{loc} (X,\sfd,\mm)$) such that
\begin{equation}\label{eq:Sobolev}
\int |f(\gamma_1)-f(\gamma_0)| \, \d \ppi(\gamma) \leq \int \int_0^1 G(\gamma_s) |\dot{\gamma}_s| \, \d s\, \d \ppi(\gamma), \quad \forall\, q\text{\rm-test plan } \ppi. 
\end{equation}
In this case, $G$ is called a $p$-weak upper gradient of $f$.
\end{definition}
Since the class of $q$-test plans  contains the one of $q'$-test plans for $q\leq q'$, we have that $\s^p_{loc}(X,\sfd, \mm)\subset \s^{p'}_{loc}(X,\sfd,\mm)$ for $p\geq p'$, and if $f\in \s^p_{loc}(X,\sfd,\mm)$ and $G$ is a $p$-weak upper gradient, then $G$ is also a $p'$-weak upper gradient.

A basic property of $p$-weak upper gradients is their lower semicontinuity w.r.t. $\mm$-a.e. convergence, in the sense that it holds
\begin{equation}
\label{eq:lscw}
\left.\begin{array}{rlr}
f_n\!\!\!\!&\to f,&\qquad\mm-a.e.,\\
f_n\!\!\!\!&\in\s^p(X,\sfd,\mm),&\qquad\forall n\in\N,\\
G_n\!\!\!\!&\textrm{ is a $p$-weak upper gradient of }f_n,&\qquad\forall n\in\N,\\
G_n\!\!\!\!&\rightharpoonup G,&\quad\textrm{in }L^p(X,\mm)
\end{array}\right\}\quad\Rightarrow\quad 
\left\{
\begin{array}{l}
f\in\s^p(X,\sfd,\mm),\\
\\
G\textrm{ is a $p$-weak}\\
\textrm{upper gradient of }f.
\end{array}
\right.
\end{equation}

Arguing as in Section 4.5 of \cite{Ambrosio-Gigli-Savare-pq}, we get that for $f\in \s^p(X,\sfd, \mm)$ (resp. $\s^p_{loc}(X,\sfd, \mm)$) there exists a minimal function $G$, in the $\mm$-a.e. sense, in $L^p(X,\mm)$ (resp. $L^p_{loc}(X, \mm)$) such that \eqref{eq:Sobolev} holds. We will denote this function by $|Df|_{w,p}.$ 

It is clear that $\s^p(X,\sfd, \mm)$ and $\s^p_{loc}(X,\sfd,\mm)$ are vector spaces and that it holds
\begin{equation}\label{eq:subAdd}
|D(\alpha f + \beta g)|_{w,p}\leq |\alpha| |Df|_{w,p}+ |\beta| |D g|_{w,p}, \quad \mm\text{\rm-a.e.}, \quad \forall \alpha,\beta \in \R.
\end{equation}
Moreover the spaces $\s^p(X,\sfd,\mm)\cap L^\infty(X,\mm)$ and $\s^p_{loc}(X,\sfd,\mm)\cap L^\infty_{loc}(X,\mm)$ are algebras on which it holds
\begin{equation}\label{eq:Algebra}
|D(fg)|_{w,p}\leq |f| |D g|_{w,p}+ |g| |D f|_{w,p}, \quad \mm\text{\rm-a.e.}.
\end{equation}
It is also possible to check that the object $|Df|_{w,p}$ is local in the sense that
\begin{equation}\label{eq:LocalityWG}
\forall f \in \s^p_{loc}(X,\sfd,\mm)\text{ it holds } |D f|_{w,p}=0, \; \mm\text{-a.e. on } f^{-1}(N), \; \forall \, N\subset \R, \; \text{s. t. } {\cal L}^1 (N)=0,
\end{equation}
and 
\begin{equation}\label{eq:LocalityWG'}
|D f|_{w,p}=|D g|_{w,p}, \quad \mm\text{-a.e. on } \{f=g\},\quad \forall\, f,g \in \s^p_{loc}(X,\sfd,\mm).
\end{equation}
Also, for $f\in \s^p(X,\sfd,\mm)$ (resp. $\s^p_{loc}(X,\sfd,\mm)$) and $\varphi:\R\to \R$ Lipschitz, the function $\varphi\circ f$ still belongs to $\s^p(X,\sfd,\mm)$ (resp. $\s^p_{loc}(X,\sfd,\mm)$) and it holds
\begin{equation}\label{eq:composition}
|D(\varphi\circ f)|_{w,p}=|\varphi' \circ f| \,|D f|_{w,p} \quad \mm\text{-a.e.}.
\end{equation}
Thanks to the locality property \eqref{eq:LocalityWG'}, a natural localized definition of Sobolev class can be given:
\begin{definition}[The Sobolev classes $\s^{p}_{loc}(\Omega)$ and $\s^p(\Omega)$]\label{def:SplocOmega}
Let $(X,\sfd,\mm)$ be as in \eqref{eq:mms}, $\Omega\subset X$ an open subset and $p\in (1,\infty)$. The space $\s^p_{loc}(\Omega)$ is the space of Borel functions $f:\Omega \to \R$ such that $f\chi \in \s^p_{loc}(X,\sfd,\mm)$ for any Lipschitz function $\chi:X\to [0,1]$ compactly supported in $\Omega$. For $f\in\s^p_{loc}(\Omega)$, the function $|D f|_{w,p}\in L^p_{loc}(\Omega,\mm)$ is defined by
\begin{equation}\label{eq:DfOmega}
|D f|_{w,p}:=|D (\nchi f)|_{w,p}, \quad \mm{\textrm{\rm -a.e. on }}\{\chi=1\},
\end{equation}
where $\chi:X \to [0,1]$ is as above  and we are thinking the function $\nchi f$ to be defined on the whole $X$, with value 0 outside $\Omega$ (thanks to \eqref{eq:LocalityWG'} this is a good definition). 

The class $\s^p(\Omega)\subset\s^p_{loc}(\Omega)$ is the one of all $f$'s such that $|Df|_{w,p}\in L^p(\Omega)$.
\end{definition}
The locality principle \eqref{eq:LocalityWG'} and the local nature of \eqref{eq:subAdd}, \eqref{eq:Algebra} and \eqref{eq:composition} imply that these latter properties are valid $\mm$-a.e. on $\Omega$ for functions $f,g\in\s^p_{loc}(\Omega)$.

\bigskip

The Sobolev space $W^{1,p}(X,\sfd,\mm)$ is defined as $W^{1,p}(X,\sfd,\mm):=\s^p(X,\sfd,\mm)\cap L^p(X,\mm)$ endowed with the norm 
\[
\|f\|^p_{W^{1,p}(X,\sfd,\mm)}:= \|f\|_{L^p(X,\mm)}^p+\||Df|_{w,p}\|_{L^p(X,\mm)}^p.
\]
$W^{1,p}(X,\sfd,\mm)$ is always a Banach space (but notice that in general  $W^{1,2}(X,\sfd,\mm)$ is not an Hilbert space). Similarly, the Sobolev spaces $W^{1,p}(\Omega)$ and $W^{1,p}_{loc}(\Omega)$ are defined as $L^p(\Omega,\mm)\cap\s^p(\Omega)$ and  $L^p_{loc}(\Omega,\mm)\cap\s^p_{loc}(\Omega)$ respectively, the former being a Banach space with the norm $\|f\|^p_{W^{1,p}(\Omega)}:= \|f\|_{L^p(\Omega,\mm)}^p+\||Df|_{w,p}\|_{L^p(\Omega,\mm)}^p$.

This definition of Sobolev space coincides with the one of Newtonian space introduced in \cite{Shanmugalingam00}, as proved in \cite{Ambrosio-Gigli-Savare-pq}. Therefore we have the following density results of Lipschitz functions  (see for instance Theorem 5.1 in \cite{BjornBjorn11} and its proof). 
\begin{theorem}[Approximation with Lipschitz functions]\label{thm:approxLip}
Let $(X,\sfd,\mm)$ be as in \eqref{eq:mms}  supporting a  $p_0$-Poincar\'e inequality \eqref{LPI} and let   $p\geq p_0$ be strictly greater than 1.

Then for every  $f \in W^{1,p}(X,\sfd,\mm)$,  there exists a  sequence $\{f_n\}_{n \in \N}$ of Lipschitz functions $W^{1,p}$-converging to $f$ and  this sequence can be chosen to satisfy   $\{f_{n+1}\neq f\}\subset \{f_n\neq f\}$ and  $\mm(\{f_n\neq f\})\to 0$ as $n\to \infty$. 

Also, if $f$ is non-negative (resp. non positive) the $f_n$'s can also be chosen non-negative (resp. non-positive) as well.

Finally, if $f$ has compact support contained in some open set $\Omega\subset X$, the $f_n$'s can be chosen so that $\cup_n\supp(f_n)$ is compact and contained in $\Omega$ as well.
\end{theorem}

Notice that, up to the present knowledge, in general the quantity $|Df|_{w,p}$ may depend on $p$. Yet,  in case $(X,\sfd,\mm)$ is doubling and supports a  $p_0$-local Poincar\'e inequality, as a consequence of Cheeger's work \cite{Cheeger00} we have that  $|Df|_{w,p}$ is independent of $p$ for $p\geq p_0$ as recalled now.
\begin{theorem}\label{thm:weakgradp}
Let $(X,\sfd,\mm)$ be as in \eqref{eq:mms} and  supporting a  $p_0$-Poincar\'e inequality \eqref{LPI}, and let $p\geq p'\geq p_0$ with $p'>1$. 

Then  every $f\in \s^p_{loc}(X,\sfd,\mm)$ also belongs to $ \s^{p'}_{loc}(X,\sfd,\mm)$ and $|D f|_{w,p}=|D f|_{w,p'}$ $\mm$-a.e.
\end{theorem}
\begin{proof} The fact that  $f \in \s^{p'}_{loc}(X,\sfd,\mm)$ and $|D f|_{w,p'}\leq |D f|_{w,p}$ are obvious. To prove that $|D f|_{w,p'}=|D f|_{w,p}$ we argue as follows. Due to the local nature of the thesis and with a truncation and cut-off argument we can assume that $f\in W^{1,p}(X,\sfd,\mm)$. Then we use Theorem \ref{thm:approxLip} and conclude as in  Corollary A.9 in \cite{BjornBjorn11} (which in turn is an application of Theorem 6.1 in \cite{Cheeger00}).   
\end{proof}

\subsection{The object $D^\pm f(\nabla g)$}\label{se:dpmfg}
From now on we will always assume that the space $(X,\sfd,\mm)$ is as in \eqref{eq:mms} and supports a $p_0$-Poicar\'e inequality \eqref{LPI} for some $p_0\in[1,\infty)$ so that, thanks to Theorem \ref{thm:weakgradp}, the weak upper gradients relative to $p,p'\geq p_0$ (if they exist) must coincide $\mm$-a.e..

In this subsection we recall the notion, introduced in \cite{Gigli12}, of differential of $f$ applied to the gradient of $g$, for $f$ and $g$ Sobolev functions and the related calculus rules. First of all notice that if $\varphi:\R \to \R^+$ is a convex function, the followings
\begin{equation}\nonumber
\liminf_{\varepsilon \downarrow 0} \frac{\varphi(\varepsilon)^p-\varphi(0)^p}{p \varepsilon \varphi(0)^{p-2}}, \quad  \limsup_{\varepsilon \uparrow 0} \frac{\varphi(\varepsilon)^p-\varphi(0)^p}{p \varepsilon \varphi(0)^{p-2}}
\end{equation} 
are actually limits as $\varepsilon\downarrow 0$, $\eps\uparrow 0$  respectively, and can be substituted by $\inf_{\varepsilon>0}$,  $\sup_{\varepsilon<0}$ respectively. Moreover they are equal to $\varphi(0) \varphi'(0^+), \varphi(0) \varphi'(0^-)$ respectively.

Now for a fixed $p\in (1,\infty)$ and for any $f,g \in \s^p_{loc}(\Omega)$, observe that \eqref{eq:subAdd} ensures that the map $\varepsilon \mapsto |D(g+\varepsilon f)|_{w,p}$ is convex in the sense that
\begin{equation}\label{eq:convexity}
|D(g+((1-\lambda)\varepsilon_0+\lambda \varepsilon_1)f)|_{w,p}\leq (1-\lambda)|D(g+\varepsilon_0 f)|_{w,p}+ \lambda |D(g+\varepsilon_1 f)|_{w,p}\quad \mm\text{-a.e.}
\end{equation}
for any $\lambda \in [0,1],\varepsilon_0,\varepsilon_1 \in \R$.   

\begin{definition}[$D^{\pm}f(\nabla g)$]\label{def:DfDg} 
Let $(X,\sfd,\mm)$ be as in \eqref{eq:mms} supporting  a $p_0$-Poincar\'e inequality \eqref{LPI},, $\Omega \subset X$ an open subset, $p\geq p_0$ strictly greater than 1 and   $f,g \in \s^p_{loc}(\Omega)$. The functions $D^{\pm}f(\nabla g):\Omega \to \R$ are $\mm$-a.e. defined by
\[
\begin{split}
D^+ f(\nabla g)&:=\liminf_{\varepsilon \downarrow 0} \frac{|D(g+\varepsilon f)|^p_{w}-|D g|^p_{w}}{p \varepsilon |D g|^{p-2}_{w}},  \\
D^- f(\nabla g)&:=\limsup_{\varepsilon \uparrow 0} \frac{|D(g+\varepsilon f)|^p_{w}-|D g|^p_{w}}{p \varepsilon |D g|^{p-2}_{w}}.
\end{split}
\]
on $\{|D g|_{w} \neq 0 \}$, and are taken $0$ by definition on $\{\weakgrad g =0 \}$.
\end{definition}

The initial discussion ensures that the limits in the definitions exist, do not depend on $p$, and can be substituted by $\inf_{\varepsilon>0}$ ($\sup_{\varepsilon<0}$ respectively). See the introduction for a discussion about the choice of the notation.

Throughout the paper, the expression $D^{\pm}f(\nabla g) |D g|_{w}^{p-2}$ on the set $\{|D g|_{w}=0\}$ will be taken 0 by definition. In this way it will always holds
\begin{eqnarray}
D^+ f(\nabla g)\, |D g|_{w}^{p-2}&=&\inf_{\varepsilon > 0} \frac{|D(g+\varepsilon f)|^p_{w}-|D g|^p_{w}}{p \varepsilon }, \nonumber \\
D^- f(\nabla g)\, |D g|_{w}^{p-2}&=&\sup_{\varepsilon < 0} \frac{|D(g+\varepsilon f)|^p_{w}-|D g|^p_{w}}{p \varepsilon }, \nonumber 
\end{eqnarray}
the equalities being intended $\mm$-a.e..

Notice that the inequality $|D(g+\varepsilon f)|_{w}\leq |D g|_{w}+|\varepsilon ||D f|_{w}$ yields
\begin{equation}\label{Schwartz}
|D^{\pm}f(\nabla g)| \leq |D f|_{w}\;|D g|_{w}, \quad \mm\textrm{-a.e. on } \Omega,
\end{equation}
and in particular  if $f,g \in \s^p(\Omega)$ (respectively $f,g \in \s^p_{loc}(\Omega)$) then $D^{\pm}f(\nabla g) |D g|_{w}^{p-2} \in L^1(\Omega)$ (respectively $D^{\pm}f(\nabla g) |D g|_{w}^{p-2} \in L^1_{loc}(\Omega)$).

The convexity of $f \mapsto |D(g+f)|_{w}^p$ gives 
\begin{equation}\label{eq:D-<D+}
D^-f(\nabla g) \leq D^+ f (\nabla g), \quad  \mm\textrm{-a.e. on } \Omega.
\end{equation}
Also, from the definition it directly follows that
\begin{equation}\label{eq:sign}
D^+ (-f)(\nabla g)=D^+f(\nabla(-g))=-D^{-}f(\nabla g)\quad \mm \textrm{-a.e. on } \Omega,
\end{equation}
and that
\begin{equation}\label{eq:Dg2}
D^{\pm}g(\nabla g)=|D g|_{w}^2 \quad  \mm \textrm{-a.e. on } \Omega.
\end{equation}
The locality properties \eqref{eq:LocalityWG}, \eqref{eq:LocalityWG'} of the weak gradients also imply that
\begin{equation}\label{eq:LocDfDg}
D^{\pm}f(\nabla g)=0, \; \mm \textrm{-a.e.   on }  f^{-1}(N)\cup g^{-1}(N), \; \forall \, N \subset \R\; \text{\rm such that } {\cal L}^1(N)=0,
\end{equation}
and that
\begin{equation}\label{eq:LocDfDg'}
D^{\pm}f(\nabla g)=D^{\pm}\tilde{f}(\nabla \tilde{g}), \quad \mm \textrm{-a.e.    on } \{f=\tilde{f}\}\cap \{g=\tilde{g}\}.
\end{equation}

The quantity $D^{\pm} f(\nabla g)$ satisfies also the useful homogeneity, convexity and semicontinuity properties stated in the next proposition (for the proof see Proposition 3.2 in \cite{Gigli12}).

\begin{proposition}[Basic properties of  $D^{\pm}f(\nabla g)$]\label{prop:ConvHomog}
Let $(X,\sfd,\mm)$ be as in \eqref{eq:mms} satisfying a $p_0$-Poincar\'e inequality \eqref{LPI},  $\Omega \subset X$ an open subset, $p\geq p_0$ strictly greater than 1 and $g\in \s^p_{loc}(\Omega)$.
Then the map
$$\s^p_{loc}(\Omega)\ni f \quad \mapsto\quad  D^+ f(\nabla g),  $$
is positively $1$-homogeneous and convex in the $\mm$-a.e. sense, i.e.
$$D^+((1-\lambda) f_0+ \lambda f_1)(\nabla g)\leq (1-\lambda) D^+ f_1(\nabla g)+\lambda D^+f_2(\nabla g), \quad \mm-a.e. \text{ on }\Omega, \ \forall\lambda\in[0,1] .$$
Similarly, 
$$\s^p_{loc}(\Omega)\ni f \quad \mapsto \quad D^- f(\nabla g),  $$
is positively 1-homogeneous and concave in the $\mm$-a.e. sense.

Moreover these maps are 1-Lipschitz in the following sense:
\begin{equation}
\label{eq:unolip}
|D^\pm f_1(\nabla g)-D^\pm f_2(\nabla g)|\leq |D(f_1-f_2)|_{w}|Dg|_{w} \qquad \mm-a.e.\quad\forall f_1,f_2 \in \s^p_{loc}(\Omega).
\end{equation}
Conversely, for given $f\in\s^p_{loc}(\Omega)$, it holds
\begin{equation}
\label{eq:semicont}
\begin{split}
\s^p_{loc}(\Omega)\ni g\qquad&\mapsto\qquad D^+f(\nabla g)\textrm{ is positively 1-homogeneous and upper semicontinuous},\\
\s^p_{loc}(\Omega)\ni g\qquad&\mapsto\qquad D^-f(\nabla g)\textrm{ is positively 1-homogeneous and lower semicontinuous},
\end{split}
\end{equation}
where upper semicontinuity is intended as follows: if $\Omega'\subset\Omega$ is open and $(g_n)\subset \s^p(\Omega')$ are such that $f,g\in\s^p(\Omega')$ and $\int_{\Omega'}\weakgrad{(g-g_n)}^p\,\d\mm\to 0$, then it holds
\[
\lims_{n\to\infty}\int_{\Omega'} D^+f(\nabla g_n)\weakgrad{g_n}^{p-2}\,\d\mm\leq \int_{\Omega'} D^+f(\nabla g)\weakgrad{g}^{p-2}\,\d\mm,
\]
similarly for lower semicontinuity.
\end{proposition}
We shall also need the following calculus rules proved in Section 3.3 of \cite{Gigli12} (notice that these properties were proved for functions in $\s^p_{loc}(X,\sfd,\mm)$, but due to the  locality properties of the statements, they directly apply to functions in $\s^p_{loc}(\Omega)$).
\begin{proposition}[Chain rules]\label{Prop:ChainRule}
Let $(X,\sfd,\mm)$ be as in \eqref{eq:mms} satisfying a $p_0$-Poincar\'e inequality \eqref{LPI}, $\Omega \subset X$ an open subset, $p\geq p_0$ strictly greater than 1 and  $f,g\in \s^p_{loc}(\Omega)$.

Then for any $\varphi:\R\to \R$ Lipschitz it holds
\begin{equation}\label{eq:ChainRulef}
D^{\pm}(\varphi\circ f)(\nabla g)=(\varphi'\circ f) \; D^{\pm \sign(\varphi'\circ f)} f (\nabla g), \quad \mm\text{\rm-a.e.  on }\Omega,
\end{equation}
where the right hand side is taken $0$ by definition at points $x$ where $\varphi$ is not differentiable at $f(x)$.

Similarly, for any $\psi:\R\to \R$  Lipschitz it holds
\begin{equation}\label{eq:ChainRuleg}
D^{\pm}f(\nabla(\psi \circ g))=(\psi'\circ g) \; D^{\pm \sign(\psi'\circ g)} f (\nabla g), \quad  \mm\text{\rm-a.e.  on }\Omega,
\end{equation}
where the right hand side is taken $0$ by definition at points $x$ where $\psi$ is not differentiable at $g(x)$.
\end{proposition}

\begin{proposition}[Leibniz rule]\label{prop:Leibniz}
Let $(X,\sfd,\mm)$ be as in \eqref{eq:mms} satisfying a $p_0$-Poincar\'e inequality \eqref{LPI},  $\Omega \subset X$ an open subset, $p\geq p_0$ strictly greater than 1, $f_1,f_2\in \s^p_{loc}(\Omega)\cap L^\infty_{loc}(\Omega)$ and  $g\in \s^p_{loc}(\Omega)$.

Then $\mm$-a.e. on ${\Omega}$  it holds 
\begin{eqnarray}
D^+(f_1f_2)(\nabla g)&\leq& f_1 D^{s_1} f_2(\nabla g)+f_2 D^{s_2} f_1(\nabla g), \nonumber\\
D^-(f_1f_2)(\nabla g)&\geq& f_1 D^{-s_1} f_2(\nabla g)+f_2 D^{-s_2} f_1(\nabla g), \nonumber
\end{eqnarray}
with $s_i:=\sign f_i$, $i=1,2$.
\end{proposition}
\begin{remark}\label{re:noleibgrad}{\rm
We recall that on a smooth Finsler manifold the Leibniz rule for gradients
\[
\nabla(g_1g_2)=g_1\nabla g_2+g_2\nabla g_1,
\]
holds for any couple of smooth functions $g_1,g_2$ if and only if the manifold is Riemannian. Hence in the general metric-measure theoretic framework in order to obtain a Leibniz rule valid for gradients we need to make an assumption which resemble `the norms on the tangent spaces come from scalar products'. Such assumption will be called infinitesimal Hilbertianity and  discussed below.
}\fr\end{remark}
Finally let us recall the definition of \emph{$q$-infinitesimally strictly convex} and \emph{infinitesimally Hilbertian} spaces  and related calculus rules.
\begin{definition}[$q$-infinitesimally strictly convex spaces]\label{def:InfStConv}
Let $(X,\sfd,\mm)$ be as in \eqref{eq:mms} satisfying a $p_0$-Poincar\'e inequality \eqref{LPI}, $\Omega \subset X$ an open subset, $p\geq p_0$ strictly greater than 1 and $q$ the  conjugate exponent. We say that $(X,\sfd,\mm)$ is \emph{$q$-infinitesimally strictly convex} if 
\begin{equation}\label{eq:InfStConv}
\int D^+ f(\nabla g)\, |D g|_{w}^{p-2} \d \mm = \int D^- f(\nabla g)\, |D g|_{w}^{p-2} \d \mm, \quad \forall f,g \in \s^p(X,\sfd,\mm).
\end{equation}
\end{definition}
From the inequality \eqref{eq:D-<D+}, we get that the integral equality \eqref{eq:InfStConv} is equivalent to the pointwise one:
\begin{equation}\label{eq:InfStConvPoint}
D^+f(\nabla g)=D^- f (\nabla g), \quad \mm\text{-a.e.,} \quad \forall \, f,g \in \s^p_{loc}(X,\sfd,\mm). 
\end{equation}
Therefore if $(X,\sfd,\mm)$ is $q$-infinitesimally strictly convex, then for every open subset $\Omega$ 
\begin{equation}\label{eq:InfStConvPointOmega}
D^+f(\nabla g)=D^- f (\nabla g), \quad \mm\textrm{-a.e. on }\Omega,  \quad \forall \, f,g \in \s^p_{loc}(\Omega). 
\end{equation}
Notice that in presence of a $p_0$-Poincar\'e inequality the identity $D^+f(\nabla g)=D^-(\nabla g)$ is observed for any $f,g\in \s^{p_0}_{loc}(\Omega)$, as shown by the following proposition.
\begin{proposition}\label{prop:freddo}
Let $(X,\sfd,\mm)$ be as in \eqref{eq:mms} and supporting a $p_0$-Poincar\'e inequality \eqref{LPI}, $p_0>1$. Let $p\geq p_0$ and $q$ be the conjugate exponent. Assume that $(X,\sfd,\mm)$ is $q$-infinitesimally strictly convex. Then for every $\Omega\subset X$ open and any $f,g\in\s^{p_0}_{loc}(\Omega)$ it holds
\[
D^+f(\nabla g)=D^-f(\nabla g),\qquad\mm\text{\rm -a.e. on }\Omega.
\]
In other words, the space is also $q_0$-infinitesimally strictly convex, $q_0$ being the conjugate exponent of $p_0$.
\end{proposition}
\begin{proof}
With a truncation and cut-off argument and by the local nature of the statement we can assume that $f,g\in W^{1,p_0}(X,\sfd,\mm)$. By the approximation result given in  Theorem \ref{thm:approxLip} we can find two sequences $(f_n)$, $(g_n)$ of Lipschitz functions converging to $f,g$ in the $W^{1,p_0}$-norm and such that
\[
\begin{split}
\lim_{n\to\infty}\mm(\{f_n\neq f\})&=0,\\
\lim_{n\to\infty}\mm(\{g_n\neq g\})&=0.
\end{split}
\]
Since $f_n,g_n$ are Lipschitz, they belong to $\s^p_{loc}(X,\sfd,\mm)$ and therefore by the identity \eqref{eq:InfStConvPoint}  we know that
\begin{equation}
\label{eq:caffeino}
D^+f_n(\nabla g_n)=D^-f_n(\nabla g_n),\qquad\mm\text{-a.e.}.
\end{equation}
By the locality property \eqref{eq:LocDfDg'}  we have
\[
D^+f(\nabla g)=D^+f_n(\nabla g_n),\qquad D^-f(\nabla g)=D^-f_n(\nabla g_n),\qquad\mm\textrm{-a.e.\ on } \{f_n=f\}\cap\{g_n=g\},
\]
and the conclusion follows letting $n\to\infty$ in \eqref{eq:caffeino}.
\end{proof}
Due to this proposition, on spaces satisfying a $p_0$-Poincar\'e inequality when dealing with infinitesimal strict convexity will directly assume that the space is $q_0$-infinitesimally strictly convex. In this case, the common value of $D^+f(\nabla g)$ and $D^-f(\nabla g)$ will be denoted by $Df(\nabla g)$.
\begin{remark}[Calculus rules in the infinitesimally strictly convex case]\label{rem:ISC}{\rm
Let $(X,\sfd,\mm)$ be as in \eqref{eq:mms}, supporting a $p_0$-Poincar\'e inequality, $p_0>1$, and  $q_0$-infinitesimally strictly convex space, $q_0$ being the conjugate exponent of $p_0$. Then  for any $g \in \s^{p_0}_{loc}(\Omega)$ the map 
$$\s^{p_0}_{loc}(\Omega)\ni f \mapsto Df (\nabla g), $$
is linear $\mm$-a.e., i.e.
\begin{equation}
\label{eq:lindif}
D(\alpha_1 f_1+\alpha_2 f_2)(\nabla g)= \alpha_1 Df_1(\nabla g)+\alpha_2 Df_2(\nabla g) \quad \mm\text{-a.e. on }\Omega,
\end{equation}
for any $f_1,f_2 \in \s^p_{loc}(\Omega)$, $\alpha_1, \alpha_2 \in \R$.

Furthermore,   under the same assumptions of Propositions   \ref{Prop:ChainRule} and \ref{prop:Leibniz} the chain rules and the Leibniz rule read as
\begin{equation}
\label{eq:calcconv}
\begin{split}
D(\varphi\circ f)(\nabla g)&=\varphi'\circ f \; D f (\nabla g),  \\
D f (\nabla(\varphi\circ g))&= \varphi'\circ g \; D f (\nabla g), \\
D(f_1 f_2)(\nabla g)&= f_1 D f_2 (\nabla g)+ f_2 D f_1 (\nabla g),
\end{split}
\end{equation}
these equalities being intended $\mm$-a.e..
}\fr\end{remark} 
\begin{definition}[Infinitesimally Hilbertian spaces]\label{def:InfHilbert}
Let  $(X,\sfd,\mm)$ be as in \eqref{eq:mms}. We say that $(X,\sfd,\mm)$ is infinitesimally Hilbertian if the map
\[
\s^2(X,\sfd,\mm)\ni f\qquad\mapsto\qquad\int\weakgrad f^2\,\d\mm,
\] 
satisfies the parallelogram rule.  
\end{definition} 
The crucial property of infinitesimal Hilbertian spaces is that not only $D^+f(\nabla g)=D^-f(\nabla g)$, but that these expressions are also symmetric in $f,g$: this is the content of the next proposition, for the proof see Proposition 4.20 in \cite{Gigli12} (see also Section 4.3 in \cite{Ambrosio-Gigli-Savare11bis}).
\begin{proposition}\label{pro:InfHilbSymm}
Let $(X,\sfd,\mm)$ be as in \eqref{eq:mms}. Then it is an infinitesimally Hilbertian space if and only for every $\Omega\subset X$ open and $f,g \in \s^2_{loc}(\Omega)$ it holds
$$D^+f(\nabla g)=D^-f(\nabla g)=D^+g(\nabla f)=D^-g (\nabla f), \quad \mm\text{\rm-a.e.\ on } \Omega.$$
\end{proposition}
If the space supports a $p_0$-Poincar\'e inequality for some $p_0\in(1,2)$, the symmetry of $Df(\nabla g)$ is observed also for functions in $\s^p_{loc}(\Omega)$ for any $p\geq p_0$, as shown in the following simple statement
\begin{proposition}\label{prop:infhil}
Let $(X,\sfd,\mm)$ be as in \eqref{eq:mms} and supporting a $p_0$-Poincar\'e inequality \eqref{LPI} for some $p_0\in(1,2)$. Assume that $(X,\sfd,\mm)$ is infinitesimally Hilbertian. Then for every $p\geq p_0$, every open set $\Omega\subset X$ and every two functions $f,g\in \s^p_{loc}(\Omega)$ it holds
\[
D^+f(\nabla g)=D^-f(\nabla g)=D^+g(\nabla f)=D^-g (\nabla f), \quad \mm\text{\rm-a.e.\ on } \Omega.
\]
\end{proposition}
\begin{proof}
Same as the proof of Proposition \ref{prop:freddo}.
\end{proof}
To highlight the symmetry of the object $Df(\nabla g)$ on infinitesimally Hilbertian spaces, we will denote it by $\nabla f \cdot\nabla g$.
\begin{remark}[Calculus rules on infinitesimally Hilbertian spaces]{\rm
Let $(X,\sfd,\mm)$ be as in \eqref{eq:mms} supporting a $p_0$-Poincar\'e inequality, $p_0\in(1,2)$, and infinitesimally Hilbertian. Then the calculus rules simplify as follows. For $\Omega\subset X$ open and $f,g\in\s^{p_0}_{loc}(\Omega)$, the linearity in \eqref{eq:lindif} yields
\begin{equation}
\label{eq:lingrad}
\begin{split}
\nabla(\alpha_1 f_1+\alpha_2 f_2)\cdot\nabla g&= \alpha_1 \nabla f_1\cdot \nabla g+\alpha_2 \nabla f_2\cdot \nabla g \qquad \forall f_1,f_2,g\in \s^{p_0}_{loc}(\Omega),\ \alpha_1,\alpha_2\in\R\\
\nabla f\cdot \nabla (\beta_1 g_1+\beta_2 g_2)&= \beta_1 \nabla f\cdot \nabla g_1+\beta_2 \nabla f\cdot \nabla g_2 \qquad \forall f,g_1,g_2\in \s^{p_0}_{loc}(\Omega),\ \beta_1,\beta_2\in\R,
\end{split}
\end{equation}
and the Leibniz rule takes the form
\begin{equation}
\label{eq:leibhil}
\begin{split}
\nabla (f_1f_2)\cdot\nabla g&=f_1\nabla f_2\cdot\nabla g+f_2\nabla f_1\cdot\nabla g,\qquad \forall f_1,f_2\in \s^{p_0}_{loc}(\Omega)\cap L^\infty_{loc}(\Omega,\mm),\ g\in\s^{p_0}_{loc}(\Omega),\\
\nabla f \cdot\nabla (g_1g_2)&=g_1\nabla f\cdot\nabla g_2+g_2\nabla f\cdot\nabla g_1,\qquad \forall f\in\s^{p_0}_{loc}(\Omega),\ g_1,g_2\in \s^{p_0}_{loc}(\Omega)\cap L^\infty_{loc}(\Omega,\mm),
\end{split}
\end{equation}
these equalities being intended $\mm$-a.e. in $\Omega$.
}\fr\end{remark}
\section{The object $\bdiv(h\nabla g) $}

\subsection{Definition}
In this subsection we define the object, fundamental for this paper, $\bdiv(h\nabla g)$ for a Sobolev function $g$ and integrable function $h$. The definition is of distributional nature and directly generalizes the one of distributional Laplacian given in \cite{Gigli12}. First of all we define the set $\test\Omega$ of  test functions in $\Omega$. 
\begin{definition}[Test functions]\label{def:testFunct}
Let $(X,\sfd, \mm)$ be as in \eqref{eq:mms} and $\Omega\subset X$ an open subset. We denote by $\Test (\Omega)$ the set of Lipschitz functions on $X$ with support  compact and contained in $\Omega$.
\end{definition} 
Notice that if $(X,\sfd,\mm)$ is as in \eqref{eq:mms} supporting a $p_0$-Poincar\'e inequality, then for $p\geq p_0$ strictly greater than 1, $g\in\s^p_{loc}(\Omega)$ and $h\in L^q(\Omega)$, $q$ being the conjugate exponent of $p$, inequality \eqref{Schwartz} and the compactness of the support of $f$ ensure that $D^\pm f(\nabla g )h$ are in $L^1(\Omega)$.
\begin{definition}[Divergence]\label{def:divergence}
Let $(X,\sfd,\mm)$ be as in \eqref{eq:mms} supporting  a $p_0$-Poincar\'e inequality \eqref{LPI}, $\Omega \subset X$ an open subset, $p\geq p_0$ strictly greater than 1 and $q$ the conjugate exponent. Let  $h \in L^q_{loc}(\Omega)$ and $g \in \s^p_{loc}(\Omega)$. We say that $h \nabla g$ is in the domain of the divergence in $\Omega$, and write $h \nabla g \in D(\bdiv, \Omega)$, if there exists a Radon measure  $\mu$ on $\Omega$ such that for any $f \in \Test(\Omega)$  it holds
\begin{equation}\label{eq:defdiv}
-\int D^{\sign(h)}f(\nabla g)h\,\d\mm\leq \int f\,\d\mu\leq-\int D^{-\sign(h)}f(\nabla g)h\,\d\mm.
\end{equation} 
In this case we write $\mu \in \bdiv(h \nabla g)\restr\Omega$.
\end{definition}
\begin{remark}[Distributional Laplacian]{\rm
If we take $h\equiv 1$, the previous definition reduces to the definition of the Laplacian given in \cite{Gigli12} (Definition 4.4). In this case we write $g\in D(\bd,\Omega)$ in place of $\nabla g\in D(\bdiv,\Omega)$ and $\mu\in \bd g\restr\Omega$ in place of $\mu\in\bdiv(\nabla g)\restr\Omega$.
}\fr\end{remark}
Notice that the divergence operator is 1-homogeneous both in $h$ and $g$, in the sense that it holds
\[\left.\begin{array}{rl}
h\nabla g&\in D(\bdiv,\Omega),\\
\mu&\in\bdiv(h\nabla g)
\end{array}\right\}
\qquad\Rightarrow\qquad \left\{
\begin{array}{rl}
\alpha h\nabla g,\, h\nabla(\alpha g)&\in D(\bdiv,\Omega),\\
\alpha\mu&\in \bdiv(\alpha h\nabla g),\,\bdiv(h\nabla(\alpha g)),
\end{array}\right.
\] 
for any $\alpha\in \R$. This can be directly checked from identities \eqref{eq:sign} and the positive 1-homogeneity statements in Proposition \ref{prop:ConvHomog}. In general, linearity in $h$ is lost due to the presence of ${\rm sign}(h)$ in the definition, and can be recovered only on infinitesimally strictly convex spaces (Proposition \ref{prop:linh}). Linearity in $g$ is false in general (essentially because gradients do not linearly depends on functions on arbitrary Finsler manifolds) and can be recovered only on infinitesimally Hilbertian spaces, see Proposition \ref{prop:ling}.

Also, the definition naturally possesses the following global-to-local property:
\begin{equation}
\label{eq:globtoloc}
\left.\begin{array}{rl}
h\nabla g&\in D(\bdiv,\Omega),\\ \mu&\in\bdiv(h\nabla g)\restr\Omega,\\ \Omega'&\subset\Omega,
\end{array}
\right\}\qquad\qquad\Rightarrow\qquad \qquad
\left\{
\begin{array}{rl}
h\nabla g&\in D(\bdiv,\Omega'),\\
 \mu\restr{\Omega'}&\in\bdiv(h\nabla g)\restr{\Omega'},
 \end{array}\right.
\end{equation}
which follows from the fact that $\test{\Omega'}$ is contained in $\test\Omega$. It is natural to expect that, being the definition of divergence of distributional nature, also a converse local-to-global property holds. We are able to achieve this result only on infinitesimally strictly convex spaces, see Proposition \ref{pro:LocGlo} and Remark \ref{re:unclear} after it for a discussion about the additional difficulties in the general case.
\begin{remark}[Potential lack of uniqueness]\label{re:nonunique}{\rm
It is easy to produce examples where the set $\bdiv(h\nabla g)$ contains more than one measure. Consider for instance the space $(\R^d,\sfd_{\|\cdot\|},\mathcal L^d)$, where $\sfd_{\|\cdot\|}$ is the distance coming from a norm $\|\cdot\|$ and $\mathcal L^d$ the Lebesgue measure. For $p>1$ let  $E_p:L^2(\R^d,\mathcal L^d)\to[0,+\infty]$ be defined by
\[
E_p(f):=\left\{\begin{array}{ll}
\displaystyle{\tfrac1p\int_{\R^d}\|Df\|_*^p\,\d\mm},&\qquad\textrm{ if the distributional differential $Df$ of $f$ is in }L^p(\R^d,\mathcal L^d),\\
+\infty,&\qquad\textrm{ otherwise},
\end{array}\right.
\]
where $\|\cdot\|_*$ is the dual norm (hence a norm on the cotangent space - i.e.  the dual - of $\R^d$) of $\|\cdot\|$. This is  the natural generalization of the $p$-energy in the normed situation. 

It is readily checked from the definition that if $f\in L^2$ is an element of the subdifferential $\partial^-E_p(g)$ of $E_p$ at some $g\in L^2$, then $f\mathcal L^d\in \bdiv(\|Dg\|_*^{p-2}\nabla g)$ in the sense of Definition \ref{def:divergence}. Hence to give an example of non-unique divergence it is sufficient to produce a norm $\|\cdot\|$ and a function $g$ such that the subdifferential of $E_p$ at $g$ contains more than one point.

To this aim, just consider a non-strictly convex norm $\|\cdot\|$ and notice that in this case the dual norm $\|\cdot\|_*$ is not differentiable. Then pick any smooth $g$ such that  for a set of $x$ of positive measure $\|\cdot\|_*$ is not differentiable at $Dg(x)$. A direct application of the definition shows that the subdifferential $\partial^-E_p(g)$ of $E_p$ at $g$ contains more than one function.
}\fr\end{remark}
\begin{remark}[Uniqueness on infinitesimally strictly convex spaces]\label{re:unique}{\rm
With the same notations and assumptions of Definition \eqref{def:divergence}, if we further assume that $(X,\sfd,\mm)$ is $q$-infinitesimally strictly convex we get that $\bdiv(h\nabla g)$ contains at most one measure $\mu$. Indeed in this case the chain of inequalities \eqref{eq:defdiv} reduces to
\[
-\int_\Omega Df(\nabla g)h\,\d\mm=\int f\,\d\mu.
\]
}\fr\end{remark}

\subsection{Calculus rules}\label{se:calcrule}
In this section we collect the basic calculus rules for the divergence. We start with the following chain rule.
\begin{proposition}[Chain rule for gradients]
Let $(X,\sfd,\mm)$ be as in \eqref{eq:mms} supporting a  $p_0$-Poincar\'e inequality \eqref{LPI}, $\Omega \subset X$ an open subset, $p\geq p_0$ strictly greater than 1 and $q$ the conjugate exponent. Let $h \in L^q_{loc}(\Omega)$, $g \in \s^p_{loc}(\Omega)$ and  let  $\varphi:\R\to \R$ be a Borel function such that for every compact subset $K \subset  \Omega$  there exists a closed  subset $I_K\subset \R$ where $\varphi$ is Lipschitz and  $\mm(g^{-1}(\R\setminus I_{K})\cap K)=0$. 

Then
\begin{equation}\label{eq:CalcRule1}
\bdiv(h\nabla(\varphi\circ g))=\bdiv(h(\varphi'\circ g) \,\nabla g),
\end{equation}
in the sense that one of  the two sides is not empty if and only if the other is,   and in this case the two sets of measures coincide.  
\end{proposition}
\begin{proof} Let $\Omega'\subset \Omega$ be an open set with compact closure contained in $\Omega$. By the assumption on $\varphi$ we know that for $\mm$-a.e. $x\in\Omega'$ it holds $g(x)\in I_{\bar\Omega'}$ and that $\varphi\restr{I_{\bar\Omega'}}$ is Lipschitz. It follows that $\varphi\circ g\in \s^p_{loc}(\Omega)$ and $\varphi'\circ g\in L^\infty_{loc}(\Omega)$. In particular, the couples of functions $(h,\varphi\circ g)$ and $(h\varphi'\circ g,g)$ satisfy the requirements asked in Definition \ref{def:divergence} and the statement makes sense.
 
Now fix  $f\in \Test(\Omega)$ and notice that the discussion just done ensures that  the chain rule \eqref{eq:ChainRuleg} is applicable to get
\[
\begin{split}
\int_\Omega D^\pm f(\nabla(\varphi\circ g))h\,\d\mm=\int_\Omega D^{\pm{\rm{sign }}(\varphi'\circ g)}f(\nabla g)\varphi'\circ g\,h\,\d\mm.
\end{split}
\]
Thus  the result follows directly from Definition \ref{def:divergence}.
\end{proof}
We now turn to the Leibniz rule for differentials. In the standard Euclidean setting it holds
\[
\div(h_1h_2\nabla g)=\nabla h_1\cdot\nabla g\,h_2+h_1\div(h_2\nabla g),
\]
and if $\R^d$ is endowed with a strictly convex norm the formula becomes
\begin{equation}
\label{eq:lettino}
\div(h_1h_2\nabla g)= Dh_1(\nabla g)\,h_2+h_1\div(h_2\nabla g).
\end{equation}
Notice that while the formula always makes sense if the functions are smooth, when we turn to the distributional notion of divergence and to  Sobolev/integrable functions, some care is needed: indeed, the term $h_1\div(h_2\nabla g)$, being the product of a Sobolev function and of a measure, in general makes no sense. In order for it to be well defined we need to assume either that $h_1$ is continuous or that the measure $\div(h_2\nabla g)$ is absolutely continuous w.r.t. Lebesgue. 

Hence the same sort of assumptions are needed in the abstract setting, and in order to prove the non-smooth analogous of \eqref{eq:lettino} we need a couple of approximation lemmas that show that the higher is the regularity of $g$ and $h$ or of $\bdiv(g \nabla h)$, the wider is the class of functions for which the integration by part rules holds.  
\begin{lemma}[Continuous and Sobolev test functions]\label{lem:IntParthg}
Let $(X,\sfd,\mm)$ be as in \eqref{eq:mms} supporting  a $p_0$-Poincar\'e inequality \eqref{LPI}, $p_0>1$, $\Omega \subset X$ an open subset and $r\geq p_0$. 

Then for every $\psi\in \s^r(X,\sfd, \mm) \cap C_c(\Omega)$ there exists a sequence $(\psi_n)\subset\test\Omega$ of uniformly bounded Lipschitz functions such that $\cup_n\supp(\psi_n)$ is compact and contained in $\Omega$, $\psi_n(x)\to\psi(x)$ for every $x\in X$ and $\|\weakgrad{(\psi-\psi_n)}\|_{L^r(X)}\to 0$ as $n\to\infty$.

In particular, the following holds. Let  $p,q,r>1$ be such that
\begin{equation}\label{hp:pqr1}
\frac{1}{p}+\frac{1}{q}+\frac{1}{r}=1, \quad p,r\geq p_0,
\end{equation}
and let $h \in L^q_{loc}(\Omega)$, $g \in \s^p_{loc}(\Omega)$ such that $h\nabla g \in D(\bdiv, \Omega)$.

Then for every $\psi \in \s^r(X,\sfd, \mm) \cap C_c(\Omega)$ and any $\mu\in \bdiv(h\nabla g)\restr\Omega$ it holds  
\[
-\int_\Omega D^{{\rm sign}(h)} \psi (\nabla g)h\,\d\mm\leq \int_\Omega \psi \,\d\mu\leq-\int_\Omega D^{-{\rm sign}(h)} \psi (\nabla g)h\,\d\mm. 
\]
\end{lemma}
\begin{proof}
Since $\mm(\supp (\psi))<\infty$ and $\psi$ is bounded, we have $\psi \in L^r(X,\mm)$ and therefore it belongs to  $ W^{1,r}(X,\sfd,\mm)$. From Theorem \ref{thm:approxLip} we know that there exists a sequence $(\psi_n)\subset W^{1,r}(X,\sfd,\mm)$ of Lipschitz functions converging to $\psi$ in $W^{1,r}(X,\sfd,\mm)$ such that
\[
\lim_{n\to\infty}\mm(\{\psi_{n}\neq \psi\})\to 0.
\] 
Define $\psi^{t,+}, \psi^{t,-}:X \to \R$ by
$$
\psi^{t,+}(x):= \inf_y \psi(y)+\frac{\sfd^2(x,y)}{2 t}, \qquad\qquad \psi^{t,-}(x):=\sup_y \psi(y)-\frac{\sfd^2(x,y)}{2t}.
$$
It is easy to check that, since $\psi \in C_b(X)$, the functions $\psi^{t,+}$ and $\psi^{t,-}$ are Lipschitz, equibounded, and it holds $\psi^{t,+}(x) \uparrow \psi(x)$,  $\psi^{t,-}(x) \downarrow \psi(x)$ as $t\downarrow 0$ for any $x \in X$ (see for instance Chapter 3 of \cite{Ambrosio-Gigli-Savare11}).

Let $\psi_{n,t}:=\min\{\max\{\psi_n, \psi^{t,+}\},\psi^{t,-}\}$ and observe that the $\psi_{n,t}$'s are Lipschitz, uniformly  bounded in $n$ and $t$, and they pointwise converge to $\psi$ as $t \to 0$ uniformly on $n$. Let $E_{n,t}:=\{\psi_{n,t}\neq\psi\}$ and notice that $E_{n,t}\subset \{\psi_n\neq \psi\}$ for any $n,t$ and thus $\lim_{n\to\infty}\mm(E_{n,t})=0$ for any $t>0$. Moreover, for any $x\in X$ it holds either $\psi_{n,t}(x)=\psi_n(x)$ or $\psi_{n,t}(x)=\psi^{t,+}(x)$ or $\psi_{n,t}(x)=\psi^{t,-}(x)$, hence from the locality property \eqref{eq:LocalityWG'}  and putting $L_{t,+}:=\Lip(\psi^{t,+})$, $L_{t,-}:=\Lip(\psi^{t,-})$ we get
\[
\begin{split}
&\int_X \weakgrad{(\psi_{n,t}-\psi)}^r\,\d\mm=\int_{E_{n,t}} \weakgrad{(\psi_{n,t}-\psi)}^r\,\d\mm\\
&\leq\int_{E_{n,t}}\weakgrad{(\psi_{n}-\psi)}^r\,\d\mm+\int_{E_{n,t}}\weakgrad{(\psi^{t,+}-\psi)}^r\,\d\mm+\int_{E_{n,t}}\weakgrad{(\psi^{t,-}-\psi)}^r\,\d\mm\\
&\leq\int_{X}\weakgrad{(\psi_{n}-\psi)}^r\,\d\mm+\int_{E_{n,t}}\big(L_{t,+}+\weakgrad\psi\big)^r\,\d\mm+\int_{E_{n,t}}\big(L_{t,-}+\weakgrad\psi\big)^r\,\d\mm,
\end{split}
\]
and therefore $\int_X \weakgrad{(\psi_{n,t}-\psi)}^r\,\d\mm\to 0$ as $n\to\infty$ for any $t>0$ (the first term goes to 0 because of the $W^{1,r}$-convergence of $(\psi_n)$ to $\psi$, the other two by the absolute continuity of the integral). Hence with a diagonalization argument we can find a sequence $t_n\downarrow0$ such that $\lim_{n\to\infty}\int_X \weakgrad{(\psi_{n,t_n}-\psi)}^r\,\d\mm=0$. By construction it also holds $\lim_{n\to\infty}\psi_{n,t_n}(x)=\psi(x)$ and $\sup_{n\in\N}|\psi_{n,t_n}|(x)<\infty$ for any $x\in X$. Hence $(\psi_{n,t_n})$ has all the desired properties except possibly the fact that $\cup_n\supp(\psi_{n,t_n})$ is a compact subset of $\Omega$. To get also this, just replace $\psi_{n,t_n}$ with $\nchi\psi_{n,t_n}$, where $\nchi\in\test\Omega$ is any function identically 1 on $\supp(\psi)$. It is immediate to check that this new sequence has all the desired properties.

The second part of the statement is a simple consequence of the first, indeed for $\psi\in\s^r(X,\sfd,\mm)\cap C_c(\Omega)$, let $(\psi_n)\subset\test\Omega$ be as given by the first part of the statement  and notice that by definition of $\mu$ it holds
\begin{equation}
\label{eq:divano}
-\int_\Omega D^{{\rm sign}(h)} \psi_n (\nabla g)h\,\d\mm\leq \int_\Omega \psi_n \,\d\mu\leq-\int_\Omega D^{-{\rm sign}(h)} \psi_n (\nabla g)h\,\d\mm,\qquad\forall n\in\N.
\end{equation}
Now conclude observing that the pointwise convergence of $\psi_n$ to $\psi$ and the fact that the $\psi_n$'s are uniformly bounded grant, via the dominated convergence theorem, that
\[
\lim_{n\to\infty}\int_\Omega\psi_n\,\d\mu=\int_\Omega\psi\,\d\mu,
\]
and that the validity of $\lim_{n\to\infty}\int_X\weakgrad{(\psi_n-\psi)}^r\,\d\mm=0$ together with \eqref{eq:unolip} give
\[
\lim_{n\to\infty}\int_\Omega D^{\pm\sign(h)}\psi_n(\nabla g)h\,\d\mm=\int_\Omega D^{\pm\sign(h)}\psi(\nabla g)h\,\d\mm.
\]
Thus we can pass to the limit in \eqref{eq:divano} and get the thesis.
\end{proof}

\begin{remark}\label{Rem:IntPartDelta}{\rm Taking $h\equiv 1$ in Lemma \ref{lem:IntParthg}, under the same assumptions on $(X,\sfd,\mm)$,  we get that if $p,r\geq p_0$,  are conjugate exponents, $g \in \s^p_{loc}(\Omega)\cap D(\bd, \Omega)$, then  for every $\psi \in \s^r(X,\sfd, \mm) \cap C_c(\Omega)$ and $\mu\in\bd g\restr\Omega$ it holds  
$$-\int_\Omega D^{+} \psi (\nabla g)\,\d\mm\leq \int_\Omega \psi \,\d\mu\leq-\int_\Omega D^{-} \psi (\nabla g)\,\d\mm,$$
which generalizes Lemma 4.25 in \cite{Gigli12} to the $\s^p(\Omega)$ framework, under dubling\&Poincar\'e assumptions.
}\fr\end{remark}

\begin{lemma}[Absolutely continuous measures in $\bdiv(h\nabla g)$ and test functions in $W^{1,r}(\Omega)$]\label{lem:IntPartm}
Let $(X,\sfd,\mm)$ be as in \eqref{eq:mms}   supporting  a $p_0$-Poincar\'e inequality \eqref{LPI}, $p_0>1$ and  $\Omega \subset X$ an open subset. Let $p,q,r>1$ be such that
\begin{equation}\label{hp:pqr2}
\frac{1}{p}+\frac{1}{q}+\frac{1}{r}=1, \quad p,r\geq p_0.
\end{equation}
Consider $h \in L^q_{loc}(\Omega)$, $g \in \s^p_{loc}(\Omega)$ with $h\nabla g \in D(\bdiv, \Omega)$ and such that for  some $\mu \in \bdiv(h\nabla g)\restr\Omega$  we have $\mu\ll \mm$ with $\frac{\d\mu}{\d\mm}\in L^{r'}_{loc}(\Omega)$,  $r'$ being the  conjugate exponent of $r$.

Then for every $\psi \in W^{1,r}(X,\sfd, \mm)$  compactly  supported  in $\Omega$  it holds  
$$-\int_\Omega D^{\sign(h)} \psi (\nabla g)h\,\d\mm \leq \int_\Omega \psi \,\d\mu \leq-\int_\Omega D^{-\sign(h)} \psi (\nabla g)h\,\d\mm. $$
\end{lemma}
\begin{proof}
By Theorem \ref{thm:approxLip} we know that there exists a sequence $(\psi_n)\subset W^{1,r}(X,\sfd,\mm)$ of Lipschitz functions converging to $\psi$ in $W^{1,r}(X,\sfd,\mm)$ and such that   $\cup_n\supp( \psi_n)$ is compact and contained in $\Omega$. 

Now  notice that the $L^r(\Omega)$ convergence of $\psi_n$ to $\psi$ yields $\lim_{n\to\infty}\int_\Omega\psi_n\,\d\mu=\int_\Omega\psi\,\d\mu$ and that the validity of $\lim_{n\to\infty}\int_X\weakgrad{(\psi_n-\psi)}^r\,\d\mm\to 0$ together with \eqref{eq:unolip} grants $\lim_{n\to\infty}\int_\Omega D^{\pm\sign(h)} \psi_n (\nabla g)h\,\d\mm=\int_\Omega D^{\pm\sign(h)} \psi (\nabla g)h\,\d\mm $. Hence we can pass to the limit in
\[
-\int_\Omega D^{\sign(h)} \psi_n (\nabla g)h\,\d\mm \leq \int_\Omega   \psi_n \,\d\mu\leq-\int_\Omega D^{-\sign(h)} \psi_n(\nabla g)h\,\d\mm,
\]
and conclude.
\end{proof}
\begin{remark}\label{Rem:IntPartDelta1}{\rm
Taking $h\equiv 1$ in Lemma \ref{lem:IntPartm}, under the same assumptions on $(X,\sfd,\mm)$,  we get that if $p,r\geq p_0$,   are conjugate exponents, $g \in \s^p_{loc}(\Omega)\cap D(\bDelta, \Omega)$ and such that for  some $\mu \in  \bDelta  g\restr\Omega$  we have $\mu\ll \mm$ with $\frac{\d\mu}{\d\mm}\in L^{p}_{loc}(\Omega, \mm\restr\Omega)$, then  for every  $\psi \in W^{1,r}(X,\sfd, \mm)$  compactly  supported  in $\Omega$  it holds  
$$-\int_\Omega D^{+} \psi (\nabla g)\,\d\mm \leq  \int_\Omega \psi\,\d\mu\leq-\int_\Omega D^{-} \psi (\nabla g)\,\d\mm, $$
which generalizes Lemma 4.26 in \cite{Gigli12} to the $\s^p(\Omega)$ framework under doubling\&Poincar\'e.
}\fr\end{remark}
In the proof of the Leibniz rule for the divergence we shall also need the following simple variant of  the Leibniz rules presented in inequality \eqref{eq:Algebra} and Proposition \ref{prop:Leibniz}.
\begin{lemma}\label{le:leiblip} Let $(X,\sfd,\mm)$ be as in \eqref{eq:mms} and supporting a $p_0$-Poincar\'e inequality, $\Omega\subset X$ an open set, $q\geq  p_0$ strictly greater than 1, $f_1\in W^{1,q}_{loc}(\Omega)$ and $f_2:\Omega\to\R$ locally Lipschitz. Then $f_1f_2\in W^{1,q}_{loc}(\Omega)$ and
\begin{equation}
\label{eq:leib2}
\weakgrad{(f_1f_2)}\leq |f_1|\weakgrad{f_2}+|f_2|\weakgrad{f_1},\qquad\mm-a.e.\textrm{ on }\Omega.
\end{equation}
Furthermore, for any $g\in \s^p_{loc}(\Omega)$, $p$ being the conjugate exponent of $q$ and assumed to be greater or equal to $p_0$, it holds 
\[
\begin{split}
D^+(f_1f_2)(\nabla g)&\leq f_1 D^{s_1} f_2(\nabla g)+f_2 D^{s_2} f_1(\nabla g),\\
D^-(f_1f_2)(\nabla g)&\geq f_1 D^{-s_1} f_2(\nabla g)+f_2 D^{-s_2} f_1(\nabla g),
\end{split}
\]
$\mm$-a.e. on $\Omega$, where $s_i$ is the sign of $f_i$, $i=1,2$.
\end{lemma}
\begin{proof}
If $f_1$ is locally bounded the thesis follows from  inequality \eqref{eq:Algebra} and Proposition \ref{prop:Leibniz}. For the general case truncate $f_1$ defining $f_1^N:=\min\{N,\max\{f,-N\}\}$ and notice that
\[
\weakgrad{(f_1^Nf_2)}\leq  |f_1^N|\weakgrad{f_2}+|f_2|\weakgrad{f_1^N}\leq |f_1||Df_2|+|f_2|\weakgrad{f_1},
\]
where $|Df_2|$ is the local Lipschitz constant of $f_2$, which is locally bounded by assumption. Therefore $\weakgrad{(f_1^Nf_2)}$ is locally uniformly bounded in $L^q$ and letting $N\to+\infty$ we deduce from \eqref{eq:lscw} that $f_1f_2\in \s^q(\Omega)$. 

The second part of the statement can be deduced by the same truncation argument using the local nature of the claim.
\end{proof}
We are now ready to prove the analogous of \eqref{eq:lettino} in metric measure spaces. Notice that for general $h_1,h_2,g$ we need to assume that the space is infinitesimally strictly convex: shortly said, this is needed because otherwise there would be a sign ambiguity in the term $D^\pm h_1(\nabla g)h_2$. In the particular case where $h_1$ is of the form $\varphi\circ g$ this ambiguity disappears because thanks to the chain rule \eqref{eq:ChainRuleg} and the identity \eqref{eq:Dg2} we know that $D^\pm(\varphi\circ g)(\nabla g)=\varphi'\circ g\weakgrad g^2$. This situation will be analyzed in Proposition \ref{prop:lettino} below.
\begin{proposition}[Leibniz rule for the divergence]\label{prop:divLap}
Let $(X,\sfd,\mm)$ be as in \eqref{eq:mms} and supporting a $p_0$-Poincar\'e inequality \eqref{LPI}, $p_0>1$. Assume furthermore that it is $q_0$-infinitesimally strictly convex, $q_0$ being the conjugate exponent of $p_0$, and let $\Omega \subset X$ be an open subset.

Let $p,q\geq p_0$ and $r\in[1,\infty]$ be such that $\frac1p+\frac1q+\frac1r=1$ and $g,h_1,h_2:\Omega\to\R$ Borel functions such that $g\in \s^p_{loc}(\Omega)$, $h_1\in W^{1,q}_{loc}(\Omega)$, $h_2\in L^r_{loc}(\Omega)$. Assume also that $h_2\nabla g\in D(\bdiv,\Omega)$ and recall that due to $q_0$-infinitesimal strict convexity there is only one measure $\mu$ in $\bdiv(h_2\nabla g)\restr\Omega$ (see Remark \ref{re:unique}).

Then the following holds.
\begin{itemize}
\item[i)] Assume that $h_1\in C(\Omega)$. Then $h_1h_2\nabla g\in D(\bdiv,\Omega)$ and the measure $\tilde\mu$ defined by
\begin{equation}
\label{eq:leibdiv}
\tilde\mu=D h_1( \nabla g)h_2\mm\restr\Omega+h_1\mu,
\end{equation}
is the only measure in $\bdiv(h_1h_2\nabla g)\restr\Omega$.
\item[ii)] Assume that $\mu\ll\mm$ with $\frac{\d\mu}{\d\mm}\in L^{q'}_{loc}(\Omega)$, $q'$ being the conjugate exponent of $q$. Then $h_1h_2\nabla g\in D(\bdiv,\Omega)$ and the measure $\tilde\mu$ defined by \eqref{eq:leibdiv} is the only measure in $\bdiv(h_1h_2\nabla g)\restr\Omega$.
\end{itemize}
\end{proposition}
\begin{proof}$\ $\\
\noindent{$\mathbf{ (i)}$} Our assumptions ensure that the right hand side of \eqref{eq:leibdiv} is a Radon measure on $\Omega$, thus the statement makes sense. Pick $f\in\test\Omega$ and use the Leibniz rule in Lemma \ref{le:leiblip} to get
\[
\begin{split}
-\int_\Omega D f( \nabla g) h_1h_2\,\d\mm= -\int_\Omega  D(fh_1)(\nabla g)h_2 \,\d\mm+\int_\Omega D h_1( \nabla g)fh_2 \,\d\mm.
\end{split}
\]  
Since $fh_1\in \s^q(\Omega)\cap C_c(\Omega)$, by Lemma \ref{lem:IntParthg}  we deduce
\[
-\int_\Omega D (fh_1)(\nabla g)h_2\,\d\mm=\int fh_1\,\d\mu,
\]
and the thesis follows.

\noindent{$\mathbf{ (ii)}$} The assumption that $\mu$ is absolutely continuous w.r.t. $\mm$ with $L^{q'}_{loc}$ density together with the hypothesis  $h_1\in L^q_{loc}(\Omega)$ ensure that the rightmost term in \eqref{eq:leibdiv} is a well defined Radon measure. Since clearly $h_2D h_1( \nabla g)\in L^1_{loc}(\Omega)$, the right hand side 
of \eqref{eq:leibdiv} defines a Radon measure and the statement makes sense.

Now let $f\in\test\Omega$ be arbitrary, and apply the Leibniz rule in Lemma \ref{le:leiblip} to get
\[
\begin{split}
-\int_\Omega h_1h_2 D f( \nabla g)\,\d\mm= -\int_\Omega D(fh_1)(\nabla g)h_2\,\d\mm+\int_\Omega fh_2 Dh_1( \nabla g)\,\d\mm.
\end{split}
\]
Since $fh_1\in W^{1,q}(\Omega)$, we can apply Lemma \ref{lem:IntPartm} and get
\[
-\int_\Omega D(fh_1)( \nabla g)h_2\,\d\mm=\int fh_1\,\d\mu,
\]
which gives the thesis.
\end{proof}
\begin{proposition}[A Leibniz rule for divergence on non-inf. strictly convex spaces]\label{prop:lettino}
Let $(X,\sfd,\mm)$ be as in \eqref{eq:mms},  supporting  a $p_0$-Poincar\'e inequality \eqref{LPI}, $p_0>1$, and $\Omega \subset X$ an open subset.

Let $p\geq p_0$  and $g,h:\Omega\to\R$ Borel functions such that $g\in \s^p_{loc}(\Omega)$,  $h\in L^{\frac p{p-2}}_{loc}(\Omega)$ and $h\nabla g\in D(\bdiv,\Omega)$. Let also   $\varphi:\R\to \R$ be a Borel function such that for every compact subset $K \subset  \Omega$  there exists a closed  subset $I_K\subset \R$ where $\varphi$ is Lipschitz and   $\mm(g^{-1}(\R\setminus I_{K})\cap K)=0$. Let $\mu\in\bdiv(h\nabla g)\restr\Omega$.

Then the following holds.
\begin{itemize}
\item[i)] Assume that $g$ is continuous. Then $\varphi\circ g\, h\nabla g\in D(\bdiv,\Omega)$ and the measure $\tilde\mu$ defined by
\begin{equation}
\label{eq:leibdiv2}
\tilde\mu:=\varphi\circ g\mu+\varphi'\circ g\, h\weakgrad g^2\mm\restr\Omega,
\end{equation}
belongs to $\bdiv(\varphi\circ g\,h\nabla g)\restr\Omega$.
\item[ii)] Assume that for some $r\in[1,\infty]$ it holds $g\in L^r_{loc}(\Omega)$ and $\mu\ll\mm$ with $\frac{\d\mu}{\d\mm}\in L^{r'}_{loc}(\Omega)$, $r'$ being the conjugate exponent of $r$. Then $\varphi\circ g\, h\nabla g\in D(\bdiv,\Omega)$ and the measure $\tilde\mu$ defined by \eqref{eq:leibdiv2} belongs to $\bdiv(\varphi\circ g\,h\nabla g)\restr\Omega$.
\end{itemize}
\end{proposition}
\begin{proof}$\ $\\
\noindent{$\mathbf{(i)}$}. We claim that $\varphi\circ g$ is continuous. Indeed, for any open set $\Omega'\subset \Omega$ with compact closure in $\Omega$ there exists by assumption a closed set $I\subset \R$ such that $\varphi\restr I$ is Lipschitz and $\mm(\Omega'\setminus g^{-1}(I))=0$. Since $g$ is continuous, $g^{-1}(I)$ is closed and thus $\Omega'\setminus g^{-1}(I)$ is open. The doubling assumption ensures that all non-empty open sets have positive measure, thus it must hold $\Omega'\setminus g^{-1}(I)=\emptyset$ and the claim follows. 

In particular, the right hand side of \eqref{eq:leibdiv2} defines a Radon measure and the statement makes sense. Let $f\in\test\Omega$ and notice that the Leibniz formulas in Lemma \ref{le:leiblip} give
\begin{equation}
\label{eq:1}
D^{\sign(h)}(\varphi\circ g\,f)(\nabla g)h\leq D^{\sign(h\varphi\circ g)}f(\nabla g)h\varphi\circ g+D^{\sign(hf)}(\varphi\circ g)(\nabla g)hf,
\end{equation}
and the chain rule in \eqref{eq:ChainRulef} and the identity \eqref{eq:Dg2}  give $D^{\sign(hf)}(\varphi\circ g)(\nabla g)hf=\weakgrad g^2 \varphi'\circ g\,hf$. Therefore integrating \eqref{eq:1} and rearranging the terms we get
\[
-\int_\Omega D^{\sign(h\varphi\circ g)}f(\nabla g)h\varphi\circ g\,\d\mm\leq -\int_\Omega D^{\sign(h)}(\varphi\circ g\,f)(\nabla g)h\,\d\mm+\int_\Omega \weakgrad g^2 \varphi'\circ g\,hf\,\d\mm.
\]
We also have $\varphi\circ g\,f\in \s^p(X,\sfd,\mm)\cap C_c(\Omega)$, thus we can apply Lemma \ref{lem:IntParthg} and deduce
\[
-\int_\Omega D^{\sign(h\varphi\circ g)}f(\nabla g)h\varphi\circ g\,\d\mm\leq \int_\Omega \varphi\circ g\,f\,\d\mu+\int_\Omega \weakgrad g^2 \varphi'\circ g\,hf\,\d\mm.
\]
With similar arguments we obtain
\[
-\int_\Omega D^{-\sign(h\varphi\circ g)}f(\nabla g)h\varphi\circ g\,\d\mm\geq \int_\Omega \varphi\circ g\,f\,\d\mu+\int_\Omega \weakgrad g^2 \varphi'\circ g\,hf\,\d\mm,
\]
so the proof is complete.\\
\noindent{$\mathbf{(ii)}$}. The proof follows the same arguments just used, with the help of Lemma \ref{lem:IntPartm} instead of Lemma \ref{lem:IntParthg}. We omit the details.
\end{proof}
We now discuss a result about existence of the divergence and comparison estimate. The statement is analogous to the classical one valid in $\R^d$ `a distribution which has a sign is a measure'. The arguments for the proof closely follow those of Proposition 4.12 in \cite{Gigli12} for existence and comparison result for the Laplacian. 
\begin{proposition}[Existence of the divergence and comparison]\label{prop:comparison}
Let $(X,\sfd, \mm)$ be as in \eqref{eq:mms} supporting a $p_0$-Poincar\'e inequality \eqref{LPI}, $\Omega\subset X$ an open subset, $p,q>1$ conjugate exponents with $p\geq p_0$. Let $g\in \s^p_{loc}(\Omega)$, $h\in L^q_{loc}(\Omega)$ and assume  that there exists a Radon  measure $\tilde\mu$ on $\Omega$ such that for any  $f \in \test\Omega$ non negative  it holds
\begin{equation}\label{eq:comparison}
-\int_{\Omega} D^{-\sign(h)}f(\nabla g)h\,\d\mm\leq \int_\Omega f\,\d\tilde\mu.
\end{equation}
Then $h\nabla g\in D(\bdiv,\Omega)$ and for any $\mu\in\bdiv(h\nabla g)\restr\Omega$ it holds $\mu\leq\tilde\mu$.
\end{proposition}
\begin{proof}
Combining assumption \eqref{eq:comparison} and  Definition \ref{def:divergence}, it is clear that if $\bdiv(h\nabla g)$ is not empty and $\mu \in \bdiv(h\nabla g)\restr\Omega$ then $\mu\leq \tilde{\mu}$, therefore in order to conclude it is enough to construct a measure $\mu \in \bdiv(h\nabla g)\restr\Omega$.

Let us consider the real valued map $\Test (\Omega)\ni f \mapsto T(f):=-\int_{\Omega} D^{-\sign(h)}f(\nabla g)h\,\d\mm$. A simple application of  Proposition \ref{prop:ConvHomog} shows that it satisfies
\begin{eqnarray}
T(\lambda f)&=&\lambda T(f), \qquad \qquad \forall f \in \Test(\Omega),\ \lambda \geq 0, \nonumber\\ \nonumber
T(f_1+f_2)&\leq& T(f_1)+T(f_2) \quad \forall f_1,f_2 \in \Test(\Omega),
\end{eqnarray}
(use the convexity of $f\mapsto D^+f(\nabla g)$ on $\{h>0\}$ and the concavity of $f\mapsto D^-f(\nabla g)$ on $\{h<0\}$).

Hence by the Hann-Banach Theorem there exists a linear map $L:\Test(\Omega)\to \R$ such that $L(f)\leq T(f)$ for any $f \in \Test(\Omega)$.  By \eqref{eq:sign}  we get
\begin{equation}\label{eq:Lf}
-\int_{\Omega} D^{\sign(h)}f(\nabla g)h\,\d\mm \leq L(f) \leq -\int_{\Omega} D^{-\sign(h)}f(\nabla g)h\,\d\mm, \quad \forall f \in \Test(\Omega).
\end{equation}
If we show that $L$ can be represented as an integral w.r.t some measure $\mu$ then the proof is complete. The  assumption \eqref{eq:comparison} together with \eqref{eq:Lf} implies that
\begin{equation}\label{eq:1provv}
\int_\Omega f \, \d \tilde{\mu}- L(f)\geq 0, \quad \forall f \in \test\Omega,\ f\geq 0. 
\end{equation}
Fix a compact subset $K\subset \Omega$ and a function $\nchi_K \in \Test (\Omega)$ with $0\leq \nchi_K \leq 1$ everywhere and $\nchi_K\equiv1$ on $K$. Let $V_K\subset \Test(\Omega)$ be the set of test functions supported in $K$ and observe that for any  $f\in V_K$, the function $(\max |f|)\nchi_K+f$ belongs to $\test\Omega$ and is non-negative. Thus \eqref{eq:1provv} yields
\[
\begin{split}
L(f)&=L((\max |f|)\nchi_K+f)-L((\max |f|)\nchi_K)\\
&\leq \int_\Omega (\max |f|)\nchi_K+f\,\d\tilde\mu-(\max |f|)L(\nchi_K)\leq (\max |f|)\Big(\int_\Omega\nchi_K\,\d\tilde\mu+\tilde\mu(K)-L(\nchi_K)\Big).
\end{split}
\]
Replacing $f$ with $-f$ we get
\[
|L(f)|\leq (\max |f|)\Big(\int_\Omega\nchi_K\,\d\tilde\mu+\tilde\mu(K)-L(\nchi_K)\Big),
\]
therefore $L:V_K\to \R$ is continuous w.r.t. the uniform norm. By the density of the Lipschitz functions in $\sup$ norm, the map $L$ can be therefore uniquely extended to a linear bounded operator on the space $C_K\subset C(X)$ of continuous functions supported in $K$. Since $K\subset  \Omega$ was an arbitrary compact subset, by the Riesz Theorem  we get that there exists a Radon measure $\mu$ on $\Omega$ representing $L$, i.e. such that
\begin{equation}\label{eq:riesz}
L(f)=\int_{\Omega} f \; \d \mu, \quad \forall  f \in \Test(\Omega).
\end{equation}
Combining \eqref{eq:Lf} and  \eqref{eq:riesz} we conclude that $h \nabla g \in D(\bdiv, \Omega)$ and $\mu \in \bdiv(h \nabla g)\restr\Omega$.
\end{proof}
\begin{remark}{\rm
In the proof of this existence result we used the Hahn-Banach theorem which in turn to be proved needs some form of Axiom of Choice. Yet, with the same assumptions and notations of the proposition, if we further assume that the space is $q$-infinitesimally strictly convex, $q$ being the conjugate exponent of $p$, we directly obtain that the map $T$ built in the proof is linear. Thus the argument goes on without any use of Choice. Our only application of Proposition \ref{prop:comparison} will be in Corollary \ref{cor:SupSolSIC}, where infinitesimal strict convexity will be assumed.
}\fr\end{remark}
Next we show that on infinitesimally strictly convex spaces the divergence is linear in $h$ and local.
\begin{proposition}[Linearity in $h$]\label{prop:linh}
Let $(X,\sfd, \mm)$ be as in \eqref{eq:mms}, supporting a $p_0$-Poincar\'e inequality, $p_0>1$, and assume also it is $q_0$-infinitesimally striclty convex, where $q_0$ is the conjugate exponent of $p_0$. Let  $p\geq p_0$, $\Omega\subset X$ an open subset $g\in\s^p_{loc}(\Omega)$ and $h_1,h_2\in L^q_{loc}(\Omega)$, $q$ being the conjugate exponent of $p$. 

Assume that $h_1\nabla g,h_2\nabla g\in D(\bdiv,\Omega)$ and denote by $\mu_1,\mu_2$ respectively the only measures in $\bdiv(h_1\nabla g)\restr\Omega,\bdiv(h_2\nabla g)\restr\Omega$ (see Remark \ref{re:unique}).

Then for every $\alpha_1,\alpha_2\in\R$ it holds $(\alpha_1h_1+\alpha_2h_2)\nabla g\in D(\bdiv,\Omega)$ and the measure $\alpha_1\mu_1+\alpha_2\mu_2$ is the only one in $\bdiv((\alpha_1h_1+\alpha_2h_2)\nabla g)\restr\Omega$
\end{proposition}
\begin{proof} Just pick $f\in\test\Omega$ and recall the linearity in \eqref{eq:lindif} to get
\[
\begin{split}
-\int_\Omega Df(\nabla g)(\alpha_1h_1+\alpha_2h_2)\,\d\mm&=-\alpha_1\int_\Omega Df(\nabla g)h_1\,\d\mm-\alpha_2\int_\Omega Df(\nabla g)h_2\,\d\mm\\
&=\alpha_1\int_\Omega f\,\d\mu_1+\alpha_2\int_\Omega f\,\d\mu_2=\int_\Omega f\,\d(\alpha_1\mu_1+\alpha_2\mu_2),
\end{split}
\]
which is the thesis.
\end{proof}

\begin{proposition}[Local to Global]\label{pro:LocGlo}
Let $(X,\sfd, \mm)$ be as in \eqref{eq:mms}, supporting a $p_0$-Poincar\'e inequality \eqref{LPI}, $p_0>1$, and assume also it is $q_0$-infinitesimally strictly convex, where $q_0$ is the conjugate exponent of $p_0$. Let $\Omega\subset X$ be an open subset and $\{\Omega_i\}_{i \in I}$ a family of open subsets such that $\Omega=\cup_i \Omega_i$.

Let  $p \geq p_0$, $q$ the conjugate exponent,  $g\in \s^p_{loc}(\Omega)$, $h\in L^q_{loc}(\Omega)$ with $h\nabla g \in D(\bdiv,\Omega_i)$ for every $i \in I$. Denote by  $\mu_i$  the only element of $\bdiv(h \nabla g)\restr{\Omega_i}$ (see Remark \ref{re:unique}). Then
\begin{equation}\label{eq:mui=muj}
{\mu_i}_{|\Omega_i\cap \Omega_j}={\mu_j}_{|\Omega_i\cap \Omega_j} \quad \forall i,j \in I, 
\end{equation}
and $h \nabla g \in D(\bdiv, \Omega)$, where  the measure $\mu$ on $\Omega$ defined by 
\begin{equation}\label{eq:defmu}
\mu_{|\Omega_i}:=\mu_i, \quad \forall i \in I 
\end{equation}
is the only element of $\bdiv(h \nabla g)\restr\Omega$. 
\end{proposition}
 
\begin{proof}
Let $i,j \in I$ and $f \in \Test(\Omega_i\cap \Omega_j)$. Then by the very definition of $\bdiv(h\nabla g)$ together with the $q_0$-infinitesimal strict convexity we get
$$-\int_{\Omega_i\cap \Omega_j} f \, \d \mu_i = \int_{\Omega_i\cap \Omega_j} D f (\nabla g) h \, \d \mm = -\int_{\Omega_i\cap \Omega_j} f \, \d \mu_j$$
which gives \eqref{eq:mui=muj}. Notice that, in particular, the measure $\mu$ is well defined by \eqref{eq:defmu}.

Now fix $f \in \Test(\Omega)$. Since $\supp f $ is compact, there exists a finite subset $I_f\subset I$ of indexes such that $\supp f \subset \cup_{i \in I_f} \Omega_i$. The doubling assumption yields that closed bounded subsets of $X$ are compact and from this it is easy to see that we can build a family $\{\nchi_i\}_{i\in I_f}$ of Lipschitz  functions such that $\sum_{i \in I_f} \nchi_i \equiv 1$ on $\supp f$ and $\supp \nchi_i$ is compactly contained in $\Omega_i$ for any $i \in I_f$. Hence  $f \nchi_i \in \Test(\Omega_i)$ for any $i \in I_f$ and taking into account the linearity of the differential expressed in \eqref{eq:lindif}, we have
\begin{equation}
\label{eq:caffe}
\begin{split}
\int_{\Omega} Df(\nabla g) h \, \d \mm&=\int_{\Omega} D\Big(\sum_{i\in I_f} \nchi_i f\Big) (\nabla g) h \, \d \mm= \sum_{i\in I_f} \int_{\Omega_i} D( \nchi_i f) (\nabla g)  h \, \d \mm\\
&=- \sum_{i\in I_f} \int_{\Omega_i} \nchi_i f \, \d \mu_i =- \int_{\Omega} f \, \d \Big(\sum_{i\in I_f}\nchi_i \mu_i\Big)=- \int_{\Omega} f \, \d\mu,
\end{split}
\end{equation}
as desired. 
\end{proof}
\begin{remark}\label{re:unclear}{\rm
It is unclear to us whether an analogous of this statement holds dropping the assumption of infinitesimal strict convexity. We remark that the equality \eqref{eq:mui=muj} certainly can't hold in full generality, because the measures $\mu_i$ are not unique. Also, carrying over the same computations in \eqref{eq:caffe} lead to inequalities `with the wrong sign'. Actually, it is unclear to us whether any local-to-global property holds for the divergence on $\R^d$ endowed with the Lebesgue measure and a non-strictly convex norm. The need for infinitesimal strict convexity in this globalization result is what prevents us to prove the sheaf property of $p$-(sub/super)-minimizers in full generality, see Proposition \ref{prop:sheaf} and Remark \ref{re:mah}.
}\fr\end{remark}
We now show that on infinitesimally Hilbertian spaces the divergence is linear in $g$ and satisfies the Leibniz rule for gradients.
\begin{proposition}[Linearity in $g$]\label{prop:ling}
Let $(X,\sfd, \mm)$ be as in \eqref{eq:mms}, supporting a $p_0$-Poincar\'e inequality \eqref{LPI}, $p_0\in(1,2]$, and infinitesimally Hilbertian. Let $p\geq p_0$, $\Omega\subset X$ open $g_1,g_2\in\s^p_{loc}(\Omega)$ and $h\in L^q_{loc}(\Omega)$, $q$ being the conjugate exponent of $p$. 

Assume that $h\nabla g_1,h\nabla g_2\in D(\bdiv,\Omega)$ and denote by $\mu_1,\mu_2$ the only measures in $\bdiv(h\nabla g_1)\restr\Omega,\bdiv(h\nabla g_2)\restr\Omega$ respectively (see Remark \ref{re:unique}).

Then for every $\beta_1,\beta_2\in\R$ it holds $h\nabla(\beta_1g_1+\beta_2g_2)\in D(\bdiv,\Omega)$ and  the measure $\beta_1\mu_1+\beta_2\mu_2$ is the only measure in $\bdiv(h\nabla(\beta_1g_1+\beta_2g_2))\restr\Omega$.
\end{proposition}
\begin{proof}
It directly follows from the linearity in $g$ of $\nabla f\cdot\nabla g$ expressed in \eqref{eq:lingrad}. Indeed, fix $f\in\test\Omega$ and notice that
\[
\begin{split}
-\int_\Omega \nabla f\cdot\nabla (\beta_1g_1+\beta_2g_2)\,h\,\d\mm&=-\beta_1\int_\Omega\nabla f\cdot\nabla g_1\,h\,\d\mm-\beta_2\int_\Omega\nabla f\cdot\nabla g_2\,h\,\d\mm\\
&=\beta_1\int_\Omega f\,\d\mu_1+\beta_2\int_\Omega f\,\d\mu_2=\int f\,\d(\beta_1\mu_1+\beta_2\mu_2),
\end{split}
\]
which is the thesis.
\end{proof}
\begin{proposition}[Leibniz Rule for gradients]
Let $(X,\sfd, \mm)$ be as in \eqref{eq:mms}, supporting a $p_0$-Poincar\'e inequality \eqref{LPI}, $p_0\in(1,2)$, and infinitesimally Hilbertian. Let $p\geq p_0$, $\Omega\subset X$ open, $g_1,g_2\in\s^p_{loc}(\Omega)\cap L^\infty_{loc}(\Omega)$ and $h\in L^q_{loc}(\Omega)$, $q$ being the conjugate exponent of $p$. 

Assume that  $hg_1\nabla g_2,hg_2\nabla g_1  \in D(\bdiv, \Omega)$ and let $\mu_1,\mu_2$ be the only measures in $\bdiv(hg_1\nabla g_2)\restr\Omega,\bdiv(hg_2\nabla g_1)\restr\Omega$ respectively (see Remark \ref{re:unique}). 

Then $h\nabla(g_1g_2)\in D(\bdiv,\Omega)$ and the measure $\tilde\mu$ defined by
\[
\tilde\mu=\mu_1+\mu_2
\]
is the only measure in $\bdiv(h\nabla(g_1g_2))\restr\Omega$.
\end{proposition}
\begin{proof}
It directly follows from the Leibniz rule for gradients given in \eqref{eq:leibhil}. Indeed, choose any $f \in \test\Omega$ and notice that
\[
\begin{split}
-\int_\Omega \nabla f\cdot\nabla (g_1g_2)\,h\d\mm&=-\int_\Omega(\nabla f\cdot\nabla g_2)\,g_1h\,\d\mm-\int_\Omega(\nabla f\cdot\nabla g_1)\,g_2h\,\d\mm\\
&=\int_\Omega f\,\d\mu_1+\int_\Omega f\,\d\mu_2=\int_\Omega f\,\d(\mu_1+\mu_2),
\end{split}
\]
which is the thesis.
\end{proof}
\section{Relation with non linear potential theory}\label{se:main}
In this section we present the main results of this paper, which consist in relating the calculus tools described up to now to (sub/super)-minimizers of the $p$-energy. 
\begin{definition}[Sub/superminimizers]
Let $(X,\sfd,\mm)$ be as in \eqref{eq:mms} supporting a $p_0$-Poincar\'e inequality, $\Omega\subset X$ an open set, $p\geq p_0$ strictly greater than 1 and $g:\Omega\to\R$ a Borel function.

We say that $g$ is a $p$-superminimizer provided $g\in\s^p(\Omega)$ and
\begin{equation}
\label{eq:supermin}
\int_\Omega\weakgrad g^p\,\d\mm\leq\int_\Omega\weakgrad{(g+f)}^p\,\d\mm\qquad\forall f\in \test\Omega,\ f\geq 0.
\end{equation}
Similarly, we say that  $g$ is a $p$-subminimizer provided $g\in\s^p(\Omega)$ and
\begin{equation}
\label{eq:submin}
\int_\Omega\weakgrad g^p\,\d\mm\leq\int_\Omega\weakgrad{(g+f)}^p\,\d\mm\qquad\forall f\in \test\Omega,\ f\leq 0.
\end{equation}  
Finally, we say that $g$ is a $p$-minimizer if it is both a $p$-superminimizer and a $p$-subminimizer.
\end{definition}
Notice that $g$ is a $p$-minimizer if and only if it holds
\begin{equation}
\label{eq:pmin}
\int_\Omega\weakgrad g^p\,\d\mm\leq\int_\Omega\weakgrad{(g+f)}^p\,\d\mm\qquad\forall f\in \test\Omega.
\end{equation}
Indeed the `if' is obvious, for the `only if' we pick $f\in\test\Omega$ apply \eqref{eq:supermin} with $\max\{f,0\}$ in place of $f$, \eqref{eq:submin} with $\min\{f,0\}$ in place of $f$ and add them up to get
\[
2\int_\Omega\weakgrad g^p\,\d\mm\leq\int_\Omega\weakgrad{(g+\max\{f,0\})}^p+\weakgrad{(g+\min\{f,0\})}^p\,\d\mm,
\]
then notice that the locality property \eqref{eq:LocalityWG'} yields that $\mm$-a.e. on $\{f\geq 0\}$ it holds
\[
\begin{split}
\weakgrad{(g+\max\{f,0\})}&=\weakgrad{(g+f)},\\
\weakgrad{(g+\min\{f,0\})}&=\weakgrad{g},
\end{split}
\]
and similarly $\mm$-a.e. on $\{f\leq 0\}$ it holds
\[
\begin{split}
\weakgrad{(g+\max\{f,0\})}&=\weakgrad{g},\\
\weakgrad{(g+\min\{f,0\})}&=\weakgrad{(g+f)}.
\end{split}
\]
Observe also that thanks to the approximation result in Theorem \ref{thm:approxLip}, \eqref{eq:supermin} is equivalent to 
\[
\int_\Omega\weakgrad g^p\,\d\mm\leq\int_\Omega\weakgrad{(g+f)}^p\,\d\mm\qquad\forall f\in \s^p(X,\sfd,\mm), \ f\geq0,\ \supp(f)\subset\subset\Omega,
\]
and similarly, \eqref{eq:submin} is equivalent to
\[
\int_\Omega\weakgrad g^p\,\d\mm\leq\int_\Omega\weakgrad{(g+f)}^p\,\d\mm\qquad\forall f\in \s^p(X,\sfd,\mm), \ f\leq0,\ \supp(f)\subset\subset\Omega.
\]
Indeed, first truncate the $f$ to get a function in $W^{1,p}$, then apply Theorem \ref{thm:approxLip} and finally pass to the limit in the truncation.

We have the following result.
\begin{theorem}[PDE characterization of $p$-minimizers]\label{thm:Solution}
Let $(X,\sfd, \mm)$ be as in \eqref{eq:mms}  and supporting a $p_0$-Poincar\'e inequality \eqref{LPI}. Let $\Omega\subset X$ be an open subset, $p\geq p_0$ strictly greater than 1 and  $g \in \s^p(\Omega)$. 

Then the following are equivalent:
\begin{itemize}
\item[i)] $g$ is a $p$-minimizer.
\item[ii)] For any $f\in\test\Omega$ it holds 
\begin{equation}
\label{eq:permin}
-\int_\Omega D^+f(\nabla g)\weakgrad g^{p-2}\,\d\mm\leq 0\leq-\int_\Omega D^-f(\nabla g)\weakgrad g^{p-2}\,\d\mm.
\end{equation}
\item[iii)]  $|D g|_w ^{p-2} \nabla g \in D(\bdiv, \Omega)$ and the null measure  $0$ belongs to $\bdiv (|D g|_w ^{p-2} \nabla g)_{|\Omega}$.
\item[iv)] For any $f\in\s^p(X,\sfd,\mm)$ with support compact and contained in $\Omega$ \eqref{eq:permin} holds.
\end{itemize}
\end{theorem}
\begin{proof}$\ $\\
\noindent{$\mathbf{(i) \Rightarrow (ii)}.$}   Fix  $f \in \Test(\Omega)$ and $\varepsilon>0$. Obviously $\eps f$ has compact support in $\Omega$ and belongs to $\s^p(\Omega)$,  hence writing inequality \eqref{eq:pmin} with $\varepsilon f$ in place of $f$ and dividing by $\varepsilon p$ we get
\begin{equation}\label{eq:4.2a}
-\int_{\Omega} \frac{|D(g+\varepsilon f)|_w^p - |D g|_w^p}{\varepsilon p} \, \d \mm \leq 0,
\end{equation}
now let $\varepsilon \downarrow 0$ and recall that
\[
 D^{+}f (\nabla g) |D g|_w^{p-2}=\lim_{\varepsilon\downarrow 0} \frac{|D(g+\varepsilon f)|_w^p - |D g|_w^p}{\varepsilon p},
\]
 and that (recall the discussion before Definition \ref{def:DfDg}) it holds
\[
-\frac{|D(g- f)|_w^p - |D g|_w^p}{p}\leq \sup_{\eps\in(0,1)} \frac{|D(g+\varepsilon f)|_w^p - |D g|_w^p}{\varepsilon p}\leq \frac{|D(g+ f)|_w^p - |D g|_w^p}{ p}.
\]
Thus by the  dominated convergence theorem we can pass to the limit in \eqref{eq:4.2a} as $\varepsilon \downarrow 0$ and obtain
 \begin{equation}\label{eq:4.3a}
 -\int_{\Omega} D^+ f (\nabla g) |D g|^{p-2}_w  \, \d \mm \leq 0.
 \end{equation}  
Arguing analogously  for $\varepsilon<0$ we get 
\begin{equation}\label{eq:4.3b}
 -\int_{\Omega} D^- f (\nabla g) |D g|^{p-2}_w  \, \d \mm \geq 0,
 \end{equation}
putting together \eqref{eq:4.3a} and \eqref{eq:4.3b} we obtain 
$$-\int_{\Omega} D^+ f (\nabla g) |D g|^{p-2}_w  \, \d \mm \leq 0 \leq -\int_{\Omega} D^- f (\nabla g) |D g|^{p-2}_w \d \mm, \qquad \forall f \in \Test(\Omega),$$
which is exactly (ii).
 
\noindent{$\mathbf{(ii) \Rightarrow (i)}.$}  Let $f\in\test\Omega$ and notice that by the very definition of $D^+f(\nabla g)$ we have
\[
D^+f(\nabla g)\weakgrad g^{p-2}\leq \frac{\weakgrad{(g+f)}^p-\weakgrad g}{p},\qquad\mm\textrm{-a.e. on }\Omega.
\]
Integrating this inequality and using the assumption \eqref{eq:permin} we deduce
\begin{equation}
\label{eq:pertest}
\int_{\Omega} |D(g+ f)|_w^p - |D g|_w^p\, \d \mm \geq 0,\qquad\forall f\in\test\Omega,
\end{equation}
which is the thesis.

\noindent{$\mathbf{(ii) \Leftrightarrow (iii)}.$} Follows by the  definition of $\bdiv(\weakgrad g^{p-2}\nabla g)\restr\Omega$.

\noindent{$\mathbf{(iv) \Rightarrow (ii)}.$} Obvious consequence of the fact that any $f\in\test\Omega$ is in $\s^p(X,\sfd,\mm)$ and with support compact and contained in $\Omega$.

\noindent{$\mathbf{(ii) \Rightarrow (iv)}.$} Let $f\in\s^p(\Omega)$ be with compact support. For $N>0$ let $f^N:=\min\{\max\{f,-N\},N\}$ and notice that $f^N\in W^{1,p}(X,\sfd,\mm)$ and has support compact and contained in $\Omega$. By  the approximation Theorem \ref{thm:approxLip} we can find a sequence $(f_n)\subset\test\Omega$  such that $\|\weakgrad{(f_n-f^N)}\|_{L^p(\Omega)}\to 0$ as $n\to\infty$. Write \eqref{eq:permin} with $f_n$ in place of $f$, recall the 1-Lipschitz estimate \eqref{eq:unolip} and let $n\to\infty$ to get
\[
-\int_\Omega D^+f^N(\nabla g)\weakgrad g^{p-2}\,\d\mm\leq 0\leq-\int_\Omega D^-f^N(\nabla g)\weakgrad g^{p-2}\,\d\mm.
\]
To conclude, just let $N\to\infty$ in these inequalities.
\end{proof}
\begin{theorem}[Almost PDE characterization of $p$-sub/superminimizers]\label{thm:SuperSolution}
Let $(X,\sfd, \mm)$ be as in \eqref{eq:mms}  and supporting a $p_0$-Poincar\'e inequality \eqref{LPI}. Let $\Omega\subset X$ be an open subset, $p\geq p_0$ strictly greater than 1  and  $g \in \s^p(\Omega)$. 

Then the following are equivalent:
\begin{itemize}
\item[i)] $g$ is a $p$-superminimizer.
\item[ii)] For any $f\in\test\Omega$ non-negative it holds
\begin{equation}
\label{eq:2}
-\int_{\Omega} D^+ f (\nabla g) |D g|^{p-2}_w  \, \d \mm \leq 0.
\end{equation}
\item[iii)] For any $f\in\s^p(X,\sfd,\mm)$ non-negative and with support compact and contained in $\Omega$ inequality \eqref{eq:2} holds.
\end{itemize}
Similarly, the following are equivalent:
\begin{itemize}
\item[i')] $g$ is a $p$-subminimizer.
\item[ii')] For any $f\in\test\Omega$ non-positive it holds
\begin{equation}
\label{eq:3}
-\int_{\Omega} D^+ f (\nabla g) |D g|^{p-2}_w  \, \d \mm \leq 0.
\end{equation}
\item[iii')] For any $f\in\s^p(X,\sfd,\mm)$ non-positive and with support compact and contained in $\Omega$ inequality \eqref{eq:2} holds.
\end{itemize}
\end{theorem}
\begin{proof}$\ $\\
\noindent{$\mathbf{(i)\Rightarrow(ii)}$}. We argue exactly as in the first part of the proof of Theorem \ref{thm:Solution}. Fix a non-negative  $f \in \Test(\Omega)$ and $\varepsilon>0$. Writing inequality \eqref{eq:supermin} with $\varepsilon f$ in place of $f$ and dividing by $\varepsilon p$ we get
\[
-\int_{\Omega} \frac{|D(g+\varepsilon f)|_w^p - |D g|_w^p}{\varepsilon p} \, \d \mm \leq 0.
\]
Letting $\eps\downarrow0$ and using the  dominated convergence theorem we conclude.

\noindent{$\mathbf{(ii)\Rightarrow(i)}$}.   We follow the same arguments used in the second part of the proof of Theorem \ref{thm:Solution}. Let $f\in\test\Omega$ be non-negative and notice that by the very definition of $D^+f(\nabla g)$ we have
\[
D^+f(\nabla g)\weakgrad g^{p-2}\leq \frac{\weakgrad{(g+f)}^p-\weakgrad g^p}{p},\qquad\mm\textrm{-a.e. on }\Omega.
\]
Integrating this inequality and using \eqref{eq:2} we conclude.

\noindent{$\mathbf{(iii)\Rightarrow(ii)}$}.  Obvious consequence of the fact that any $f\in\test\Omega$ is in $\s^p(X,\sfd,\mm)$ and with support compact and contained in $\Omega$.

\noindent{$\mathbf{(ii) \Rightarrow (iii)}.$} It follows from the approximation with Lipschitz functions provided in Theorem \ref{thm:approxLip} in conjunction  with the 1-Lipschitz estimate \eqref{eq:unolip}, as in the proof of Theorem \ref{thm:Solution}.

\noindent{$\mathbf{(i') \Leftrightarrow (ii')\Leftrightarrow (iii')}.$} The conclusion comes from the very same arguments.
\end{proof}
It is important to remark that (ii) (and similarly (ii')) of the above theorem is in general not the same as requiring $\weakgrad g^{p-2}\nabla g\in D(\bdiv,\Omega)$ with $0\geq\mu$ for some/any $\mu\in \bdiv(\weakgrad g^{p-2}\nabla g)$. This is due to the fact that the object $D^+f(\nabla g)\weakgrad g^{p-2}$ appears in the left-hand-side of \eqref{eq:2}, rather than $D^-f(\nabla g)\weakgrad g^{p-2}$ which would allow to apply Proposition \ref{prop:comparison}. Yet, at least for infinitesimally strictly convex spaces a  PDE characterization of sub/superminimizers  can be obtained, as shown by the following Corollary.
\begin{corollary}[PDE characerization of $p$-sub/superminimizers on inf. strictly convex spaces]\label{cor:SupSolSIC}
Let $(X,\sfd, \mm)$ be as in \eqref{eq:mms} supporting a $p_0$-Poincar\'e inequality, $p_0>1$, and $q_0$-infinitesimally strictly convex, where $q_0$ is the conjugate exponent of $p_0$. Let $\Omega\subset X$ be an open subset, $p\geq p_0$   and  $g \in \s^p(\Omega)$. 

Then the following are equivalent:
\begin{itemize}
\item[i)] $g$ is a $p$-superminimizer.
\item[ii)] It holds $\weakgrad g^{p-2} \nabla g \in D(\bdiv, \Omega)$ and the only measure  $\mu$ in  $\bdiv (\weakgrad g^{p-2}\nabla g)\restr\Omega$ is non-positive.
\end{itemize}
Similarly, the following are equivalent:
\begin{itemize}
\item[i')] $g$ is a $p$-subminimizer.
\item[ii')] It holds $\weakgrad g^{p-2} \nabla g \in D(\bdiv, \Omega)$ and  $\mu$ in  $\bdiv (\weakgrad g^{p-2}\nabla g)\restr\Omega$ is non-negative.
\end{itemize}
\end{corollary}
\begin{proof}
Just recall that on $q_0$-infinitesimally strictly convex spaces it holds
\[
-\int_{\Omega} D^- f (\nabla g) |D g|^{p-2}_w  \, \d \mm=-\int_{\Omega} D^+ f (\nabla g) |D g|^{p-2}_w  \, \d \mm,\qquad\forall f\in \s^p(\Omega).
\]
Therefore the conclusions come from Theorem \ref{thm:SuperSolution} and  Proposition \ref{prop:comparison}.  
\end{proof}

\section{Some applications}\label{se:appl}
\subsection{Sheaf property of harmonic functions}
As a consequence of the local nature of the definition of distributional divergence,   in case of $q_0$-infinitesimally strictly convex spaces we can give a positive answer to the Open Problem 9.22 in \cite{BjornBjorn11} concerning the sheaf property of harmonic functions. More precisely, we will prove the sheaf property of $p$-sub/superminimizers, while the problems in \cite{BjornBjorn11} are stated for $p$-sub/superharmonic functions. The latter are defined in terms of comparisons with harmonic functions attaining the same value of the given function at boundaries of open sets. Like in the standard Euclidean case, there are strong connections between the two concepts, see Chapter 9 in \cite{BjornBjorn11} for an overview.
\begin{proposition}[Sheaf property of $p$-minimizers and $p$-sub/superminimizers]\label{prop:sheaf}
Let $(X,\sfd, \mm)$ be as in \eqref{eq:mms}, supporting a $p_0$-Poincar\'e inequality \eqref{LPI}, $p_0>1$, and $q_0$-infinitesimally strictly convex, where $q_0$ is the conjugate exponent to $p_0$. Let  $p\geq p_0$ and  $\{\Omega_i\}_{i \in I}$ a family of open subsets. Put $\Omega:=\cup_i\Omega_i$ and let $g \in \s^p(\Omega)$.

Then  $g$ is a $p$-minimizer  (resp. $p$-superminimizer, resp. $p$-subminimizer)  on $\Omega$ if and only if it is a $p$-minimizer (resp. $p$-superminimizer, resp. $p$-subminimizer)  on $\Omega_i$ for every $i\in I$.
\end{proposition}  
\begin{proof}
It is clear that if $g$ is a $p$-minimizer (resp. $p$-superminimizer,  resp. $p$-subminimizer)  on $\Omega$ then it is also  a $p$-minimizer (resp. $p$-superminimizer,  resp. $p$-subminimizer)  on $\Omega_i$ for every $i\in I$, thus we pass to the converse implication. 

Assume that $g$ is a $p$-superminimizer on $\Omega_i$ for every $i\in I$. Then by Corollary \ref{cor:SupSolSIC} we have $\weakgrad g^{p-2} \nabla g \in D(\bdiv, \Omega_i)$ for every $i\in I$ and the only measure $\mu_i$ in $\bdiv(\weakgrad g^{p-2}\nabla g)\restr{\Omega_i}$ satisfies $\mu_i\leq 0$. We now apply  Proposition \ref{pro:LocGlo} with $h:=\weakgrad g^{p-2}\in L^q(\Omega)$ to deduce that $\weakgrad g^{p-2} \nabla g \in D(\bdiv,\Omega)$ and that calling $\mu$ the only measure in $\bdiv(\weakgrad g^{p-2}\nabla g)\restr\Omega $ it holds $\mu_{|\Omega_i}=\mu_i$ for all $i\in I$. Hence $\mu\leq 0$ and using again Corollary \ref{cor:SupSolSIC} we conclude that $g$ is a $p$-superminimizer  on $\Omega$.

A similar argument applies to  $p$-subminimizers and   $p$-minimizers.
\end{proof}
\begin{remark}\label{re:mah}{\rm
Given that we assumed infinitesimal strict convexity to prove the sheaf property of $p$-minimizers, it is natural to question what happens if this hypothesis is dropped. We don't know. Worse than this, we don't know the answer neither for $p=2$ when the base space is $\R^2$ equipped with the Lebesgue measure and a non-strictly convex norm. The fact that this problem looks - to us - non-trivial to treat even in such a simple and concrete case, suggests that there might be additional intrinsic geometric/analytic complications when the considered space is not assumed to be infinitesimally strictly convex.
}\fr\end{remark}
\subsection{Composition of superminimizers with convex and increasing functions}
The availability of a differential calculus allows, in some case, to simplify proofs or at least to let them look closer to what they are in the standard Euclidean case. As an example we give a new proof of the fact that the composition of a superminimizer with a convex and non-increasing function is a subminimizer and some related properties, see Theorem 9.41 and Corollary 9.43 in \cite{BjornBjorn11} for a different approach to similar statements. On infinitesimally strictly convex spaces and for smooth functions $\varphi$ all the properties stated below are a consequence of the equivalence stated in Corollary \ref{cor:SupSolSIC} and the chain rule
\[
\begin{split}
\bdiv(\weakgrad{(\varphi\circ g)}^{p-2}\nabla (\varphi\circ g))=&|\varphi'\circ g|^{p-2}\varphi'\circ g\,\bdiv(\weakgrad g^{p-2}\nabla g)\\
&+(p-1)|\varphi'\circ g|^{p-2}\varphi''\circ g\weakgrad g^p\mm,
\end{split}
\]
which in turn follows from  \eqref{eq:composition}, \eqref{eq:ChainRuleg} and Proposition \ref{prop:divLap}.

Yet, on the general case we can't proceed this way for two reasons: the first is that if the space is not infinitesimally strictly convex  we don't have a PDE characterization of sub/superminimizers, the second is that if $\varphi$ is not $C^{1,1}$ the term $|\varphi'\circ g|^{p-2}\varphi''\circ g\weakgrad g^p$ in the above formula makes no sense.
\begin{proposition}\label{prop:compsub}
Let $(X,\sfd,\mm)$ be as in \eqref{eq:mms} be supporting a $p_0$-Poincar\'e inequality \eqref{LPI}. Let $\Omega\subset X$ an open set, $p\geq p_0$  strictly greater than 1 and $g\in\s^p(\Omega)$. Also, let $I\subset\R$ be a closed interval such that $\mm(g^{-1}(\R\setminus I))=0$ and $\varphi:I\to\R$ a function.

Then the following are true.
\begin{itemize}
\item[i)] If $g$ is a $p$-superminimizer and $\varphi$ is convex, Lipschitz and non-increasing, then $\varphi\circ g$ is a $p$-subminimizer
\item[ii)] If $g$ is a $p$-superminimizer and $\varphi$ is concave, Lipschitz and non-decreasing, then $\varphi\circ g$ is a $p$-superminimizer
\item[iii)] If $g$ is a $p$-subminimizer and $\varphi$ is convex, Lipschitz and non-decreasing, then $\varphi\circ g$ is a $p$-subminimizer
\item[iv)] If $g$ is a $p$-subminimizer and $\varphi$ is concave, Lipschitz and non-increasing, then $\varphi\circ g$ is a $p$-superminimizer
\end{itemize}
\end{proposition}
\begin{proof}
We will prove only (i), the proof of the other points being similar. Since $\varphi$ is Lipschitz, $\varphi\circ g\in\s^p(\Omega)$. Assume for a moment that $\varphi$ is $C^2$ with bounded second derivative, so that $\varphi'\circ g\in\s^p(\Omega)$. Let $f\in\test\Omega$ be non-positive and use the chain rules \eqref{eq:composition}, \eqref{eq:ChainRuleg} and the Leibniz rule in Lemma \ref{le:leiblip}  to get
\begin{equation}
\label{eq:a}
\begin{split}
D^+f(\nabla(\varphi\circ g))&\weakgrad{(\varphi\circ g)}^{p-2}\\
&=\big|\varphi'\circ g\big|^{p-2}\varphi'\circ g\,D^-f(\nabla g)\weakgrad{g}^{p-2}\\
&\geq  D^+(f \big|\varphi'\circ g\big|^{p-2}\varphi'\circ g)(\nabla g)\weakgrad g^{p-2}-fD^-( \big|\varphi'\circ g\big|^{p-2}\varphi'\circ g)(\nabla g)\weakgrad g^{p-2}\\
&= D^+(f \big|\varphi'\circ g\big|^{p-2}\varphi'\circ g)(\nabla g)\weakgrad g^{p-2}-(p-1)f\big|\varphi'\circ g\big|^{p-2}\varphi''\circ g\weakgrad g^p,
\end{split}
\end{equation}
where in the last equality we used first the chain rule \eqref{eq:ChainRulef} with $g$ in place of $f$ and $|\varphi'|^{p-2}\varphi'$ in place of $\varphi$, and then the identity \eqref{eq:Dg2}.

Since $f\leq 0$ and $\varphi''\geq 0$ we have
\begin{equation}
\label{eq:b}           
-(p-1)f\big|\varphi'\circ g\big|^{p-2}\varphi''\circ g\weakgrad g^p\geq 0.
\end{equation}
Also, the function $\psi(z):=|\varphi'(z)|^{p-2}\varphi'(z)$ is Lipschitz and thus by \eqref{eq:composition} we know that $\big|\varphi'\circ g\big|^{p-2}\varphi'\circ g$ is in $\s^p(\Omega)$, hence $f\big|\varphi'\circ g\big|^{p-2}\varphi'\circ g$ is in $\s^p(\Omega)$ as well and has compact support in $\Omega$ and the assumptions $f\leq 0$, $\varphi'\leq 0$ ensure that this function is non-negative. By  $(iii)$ of Theorem \ref{thm:SuperSolution} with $f\big|\varphi'\circ g\big|^{p-2}\varphi'\circ g$ in place of $f$ we deduce
\begin{equation}
\label{eq:c}
\int_\Omega D^+(f\big|\varphi'\circ g\big|^{p-2}\varphi'\circ g)(\nabla g)\weakgrad g^{p-2}\,\d\mm\geq 0.
\end{equation}
Thus in this case the thesis follows integrating \eqref{eq:a}, using \eqref{eq:b} and \eqref{eq:c} and then recalling Theorem \ref{thm:SuperSolution}.

Now we consider the general case where $\varphi$ is not necessarily $C^2$. With a simple smoothing argument we can find a sequence  $(\varphi_n)\subset C^2(I)$ such that: each $\varphi_n$ has  bounded second derivative, the $\varphi_n$'s are  uniformly Lipschitz, convex and non-increasing and satisfy
\begin{equation}
\label{eq:limitebus}
\lim_{n\to\infty}\varphi'_n(z)=\varphi'(z),\qquad\textrm{ for $\mathcal L^1\restr I$-a.e. $z$}.
\end{equation}
By what we previously proved we know that 
\begin{equation}
\label{eq:n}
\int_\Omega D^+f(\nabla(\varphi_n\circ g))\weakgrad{(\varphi_n\circ g)}^{p-2}\,\d\mm\geq 0,\qquad\forall n\in\N, \  f\in\test\Omega,\ f\leq 0.
\end{equation}
We claim that $\int_\Omega\weakgrad{(\varphi\circ g-\varphi_n\circ g)}^p\,\d\mm\to 0$ as $n\to\infty$. Indeed, the uniform Lipschitz property of the $\varphi_n$'s ensures that for some $L>0$ it holds $\weakgrad{(\varphi_n\circ g)}\leq L\weakgrad g$, so that the sequence $\weakgrad{(\varphi\circ g-\varphi_n\circ g)}^p$ is dominated, while the chain rule \eqref{eq:composition} and \eqref{eq:limitebus}  yield that  $\weakgrad{(\varphi\circ g-\varphi_n\circ g)}\to0$ $\mm$-a.e.. Thus the claim follows by the dominated convergence theorem. 

The upper semicontinuity property stated in \eqref{eq:semicont} with $\varphi_n\circ g$ in place of $g_n$ and $\varphi\circ g$ in place of $g$ together with \eqref{eq:n} gives
\[
\int_\Omega D^+f(\nabla(\varphi\circ g))\weakgrad{(\varphi\circ g)}^{p-2}\,\d\mm\geq \lims_{n\to\infty}\int_\Omega D^+f(\nabla(\varphi_n\circ g))\weakgrad{(\varphi_n\circ g)}^{p-2}\,\d\mm\geq 0,
\]
for any non-positive $f\in\test\Omega$, as desired.
\end{proof}

\subsection{Harmonicity of the Busemann function associated to a line on infi\-ni\-tesi\-mally Hilbertian $CD(0,N)$ spaces}\label{se:bus}
A crucial step in the proof of the splitting theorem on Riemannian manifolds with non negative Ricci curvature is the fact that the Busemann function associated to a line is harmonic. This is a consequence of the strong maximum principle applied to the function $\b^++\b^-$, where $\b^\pm$ are the Busemann functions associated to the respective semi-lines, which has minima along the line itself and satisfies $\Delta(\b^++\b^-)\leq 0$ (see below for the definitions).

In \cite{Gigli12} it has been proved that the Busemann function associated to a semi-line on $CD(0,N)$ spaces has non-positive Laplacian, but the proof  that it is harmonic (if associated to a line) was not completed due to the lack of a strong maximum principle. Here we complete this step relying on the fact that the strong maximum principle is indeed known to be true on doubling spaces supporting a Poincar\'e inequality and on the PDE characterization of superminimizers that we just proved.

We recall the following result proved in \cite{BjornBjorn11} (see Theorem 9.13). Notice that the formulation we are giving here is weaker than the one stated in \cite{BjornBjorn11}, but sufficient for our purposes.
\begin{theorem}[Strong maximum principle]\label{thm:maxpr}
Let $(X,\sfd,\mm)$ be as in \eqref{eq:mms} and supporting a 2-Poincar\'e inequality \eqref{LPI}. Let $g:X\to\R$ be a lower semicontinuous function in $\s^2_{loc}(X,\sfd,\mm)$ with the following property: for any $\Omega\subset X$ open with compact closure it holds
\[
\int_\Omega\weakgrad g^2\,\d\mm\leq\int_\Omega\weakgrad{(g+f)}^2\,\d\mm,\qquad\forall f\in\test\Omega,\ f\geq 0.
\]

Assume that $g$ has a minimum. Then $g$ is constant.
\end{theorem}
We will apply the strong maximum principle on $CD(0,N)$ spaces, whose definition is recalled below.

Given a complete separable metric space $(X,\sfd)$ endowed with a non negative Radon measure $\mm$ finite on bounded sets and a number $N\in(1,\infty)$, we consider the functional $\mathcal U_N$ defined on the space of probability measures with bounded support as:
\[
\mathcal U_N(\mu):=-\int\rho^{1-\frac1N}\,\d\mm,\qquad\mu=\rho\mm+\mu^s,\ \mu^s\perp\mm
\] 
\begin{definition}[$CD(0,N)$ spaces]
Let $(X,\sfd)$ be a complete separable metric space endowed with a non negative Radon measure $\mm$ finite on bounded sets and $N\in(1,\infty)$. We say that $(X,\sfd,\mm)$ is a $CD(0,N)$ space provided for every two probability measures with bounded support $\mu,\nu$ on $X$ there exists a $W_2$-geodesic $(\mu_t)$ connecting them such that
\[
\mathcal U_{N'}(\mu_t)\leq(1-t)\mathcal U_{N'}(\mu)+t\mathcal U_{N'}(\nu),\qquad\forall t\in[0,1],
\]
holds for any $N'\geq N$.
\end{definition}
The fact that we can apply the maximum principle in Theorem \ref{thm:maxpr} to $CD(0,N)$ space is ensured by the following proposition, see \cite{Lott-Villani09} and \cite{Sturm06II} for the proof of the doubling property and \cite{Rajala11} for the proof of the Poincar\'e inequality (see also \cite{Lott-Villani-Poincare} for the original argument on non-branching spaces).
\begin{proposition}[$CD(0,N)$ implies doubling and Poincar\'e]\label{prop:cddoubpoi}
Let $N\in(1,\infty)$ and $(X,\sfd,\mm)$ a $CD(0,N)$ space. Then $(\supp(\mm),\sfd,\mm)$ is doubling and supports a 1-Poincar\'e inequality (and thus a fortiori a 2-Poincar\'e inequality). 
\end{proposition}

Given a metric space $(X,\sfd)$, a curve $\gamma:[0,\infty)\to X$ is said a half line provided it holds
\[
\sfd(\gamma_t,\gamma_s)=|s-t|,\qquad\forall t,s\geq 0.
\]
The Busemann function $\b:X\to\R$ associated to an half line $\gamma$ is defined as
\[
\b(x):=\inf_{t\geq 0}\sfd(\gamma_t,x)-t=\lim_{t\to+\infty}\sfd(\gamma_t,x)-t.
\]
The fact that the $\inf$ is equal to the $\lim$ is a consequence of the triangle inequality, which also ensures that $\b$ never takes the value $-\infty$.

The following result has been proved in \cite{Gigli12}, see Proposition 5.19.
\begin{proposition}[Laplacian comparison for the Busemann function on $CD(0,N)$ spaces]\label{prop:lapbus}
Let $N\in (1,\infty)$, $(X,\sfd,\mm)$ a 2-infinitesimally strictly convex $CD(0,N)$ space and assume furthermore that $W^{1,2}(X,\sfd,\mm)$ is uniformly convex. 

Assume that there exists an half-line $\gamma:[0,\infty)\to \supp(\mm)$ and let $\b$ be the Busemann function associated to it. 

Then $\b\in D(\bd, X)$ and denoting by $\mu$ the only measure in $\bd\b$, it holds $\mu\leq 0$.
\end{proposition}
Given a metric space $(X,\sfd)$, a curve $\gamma:\R\to X$ is said to be a line provided it holds
\[
\sfd(\gamma_t,\gamma_s)=|s-t|,\qquad\forall t,s\in\R.
\]
To a line we can associate two Busemann functions $\b^+,\b^-$ according to whether the limit is taken as $t\to+\infty$ or $t\to-\infty$:
\[
\begin{split}
\b^+(x)&:=\inf_{t\geq 0}\sfd(\gamma_t,x)-t=\lim_{t\to+\infty}\sfd(\gamma_t,x)-t,\\
\b^-(x)&:=\inf_{t\geq 0}\sfd(\gamma_{-t},x)-t=\lim_{t\to+\infty}\sfd(\gamma_{-t},x)-t.
\end{split}
\]
Thanks to the maximum principle, we can obtain the following result.
\begin{theorem}[The Busemann function is harmonic on infinitesimally Hilbertian $CD(0,N)$ spaces] Let $(X,\sfd,\mm)$ be an infinitesimally Hilbertian $CD(0,N)$ space, $\gamma:\R\to X$ a line and $\b^\pm$ the Busemann functions associated to it. 

Then $\b^+=-\b^-$ on $\supp(\mm)$. In particular, denoting by $\mu^\pm$ the only measures in $\bd\b^\pm$,  it holds $\mu^+=\mu^-=0$.
\end{theorem}
\begin{proof}
By Proposition \ref{prop:lapbus} we know that $\b^+,\b^-\in D(\bd, X)$ with $\mu^+,\mu^-\leq 0$. Let $g:=\b^++\b^-$ and notice that since both $\b^+$ and $\b^-$ are Lipschitz, they both belong to $\s^2_{loc}(X,\sfd,\mm)$. Hence by Proposition \ref{prop:ling} (applied with $h\equiv1$) we have $g\in D(\bd, X)\cap\s^2_{loc}(X,\sfd,\mm)$ and the only measure $\mu$ in $\bd g$ satisfies $\mu=\mu^++\mu^-\leq 0$.

Since $g\in\s^2_{loc}(X,\sfd,\mm)$, for every open $\Omega\subset X$ with compact closure we have $g\in\s^2(\Omega)$ and the only measure in $\bd g\restr\Omega$ is $\mu\restr\Omega$. The inequality $\mu\leq 0$ gives $\mu\restr\Omega\leq 0$. By Proposition \ref{prop:cddoubpoi} we know that  $(\supp(\mm),\sfd,\mm)$ is doubling and supports a 2-Poincar\'e inequality,  therefore we can apply Corollary \ref{cor:SupSolSIC} to deduce that
\[
\int_\Omega\weakgrad g^2\,\d\mm\leq\int_\Omega\weakgrad{(g+f)}^2\,\d\mm,\qquad\forall\Omega\subset\subset X\textrm{ open},\  f\in\test\Omega,\ f\geq 0.
\]
Finally, $g$ has minimum in $\supp(\mm)$, because the triangle inequality gives $g\geq 0$ and by definition we have $g(\gamma_t)=0$ for any $t\in\R$. Hence we can apply Theorem \ref{thm:maxpr} to the space $(\supp(\mm),\sfd,\mm)$ and conclude.
\end{proof}
It is worth pointing out that there is a difference between what we are able to achieve in abstract spaces and what is true in the smooth Finsler setting. Indeed, here to apply the maximum principle we had to assume that the space is infinitesimally Hilbertian. This was needed to track the information on non-positivity of the Laplacian from $\b^\pm$ to $\b^++\b^-$. If the Laplacian is not linear, in general from $\bd \b^+\leq 0$ and $\bd\b^-\leq 0$ we can't deduce $\bd(\b^++\b^-)\leq 0$.

Yet, on smooth Finsler manifolds one has at disposal a maximum principle stronger than the one expressed in Theorem \ref{thm:maxpr}. Indeed, it is known that if $g_1,g_2$ are such that $\Delta g_1\leq 0$ and $\Delta g_2\leq 0$ ($\Delta$ being the natural, possibly nonlinear, Laplacian on the manifold) and $g_1+g_2$ has a minimum, then $g_1+g_2$ is constant, see Lemma 5.4. in \cite{GS}  and the references therein. A formulation like this is exactly what is necessary to get that the Busemann function is harmonic. However, as far as we know, such natural generalization of the maximum principle  is currently unavailable on the non-smooth setting.

\bibliographystyle{siam}
\bibliography{biblio}

\end{document}